\documentclass[reqno]{amsart}

\usepackage{amsmath}
\usepackage{amsthm}
\usepackage{amsfonts}
\usepackage{amssymb}
\usepackage{url}
\usepackage{enumerate}
\usepackage[pdftex,bookmarks=true]{hyperref}

\usepackage{xcolor}

\newcommand\R{{\mathbb{R}}}

\newcommand\Z{{\mathbf{Z}}}

\renewcommand\P{{\mathbf{P}}}

\newcommand\eps{{\varepsilon}}

\newcommand\dist{\operatorname{dist}}

\newcommand\CB{{\mathcal B}}

\newcommand\CE{{\mathcal E}}
\newcommand\CF{{\mathcal F}}
\newcommand\CG{{\mathcal G}}
\newcommand\CH{{\mathcal H}}

\newcommand\CN{{\mathcal N}}

\newcommand\ep{{\xi}}

\newcommand\LCD{\mathbf{LCD}}

\newcommand\Comp{\mathbf{Comp}}
\newcommand\Incomp{\mathbf{Incomp}}
\newcommand\supp{\mathbf{supp}}
\newcommand\spread{\mathbf{spread}}

\newcommand\LCDhat{\widehat{\mathbf{LCD}}}

\theoremstyle{plain}
 \newtheorem{theorem}{Theorem}[section]
 \newtheorem{conjecture}[theorem]{Conjecture}
 \newtheorem{question}[theorem]{Question}

 \newtheorem{heuristic}[theorem]{Heuristic}
 
 \newtheorem{fact}[theorem]{Fact}
 \newtheorem{lemma}[theorem]{Lemma}
 \newtheorem{corollary}[theorem]{Corollary}

\newtheorem{remark}[theorem]{Remark}

\theoremstyle{definition}
\newtheorem{definition}[theorem]{Definition}

\include{psfig}

\begin{document}

\title[Eigenvalues gap]{Random matrices: Tail bounds for gaps between eigenvalues}

\author{Hoi Nguyen}
\address{Department of Mathematics, The Ohio State University, Columbus OH 43210}
\email{nguyen.1261@math.osu.edu}
\thanks{H. Nguyen is supported by NSF grant DMS-1358648.}

\author{Terence Tao}
\address{Department of Mathematics, UCLA, Los Angeles CA 90095}
\email{tao@math.ucla.edu}
\thanks{T. Tao is supported by a Simons Investigator grant, the
James and Carol Collins Chair, the Mathematical Analysis \& Application Research Fund Endowment, and by NSF grant DMS-1266164.}

\author{Van Vu}
\address{Department of Mathematics, Yale University, New Haven CT 06520}
\email{van.vu@yale.edu}
\thanks{V. Vu is supported by   NSF  grant DMS-1307797  and AFORS grant FA9550-12-1-0083.}

\begin{abstract}   Gaps (or spacings)  between consecutive eigenvalues are a central topic in random matrix theory. 
The goal of this paper is to study the  tail distribution of these gaps  in various random matrix  models. 
We give the first repulsion bound for random matrices with discrete entries and the first  super-polynomial bound on the probability that a random graph 
has simple spectrum, along with several applications. 
\end{abstract}

\maketitle

\setcounter{tocdepth}{2}

\section{Introduction}\label{section:intro}
Let $M_n$ be a  random Hermitian matrix of size $n$, with eigenvalues $\lambda_1 \le \dots \le \lambda_n$. In this paper, we study the  tail distribution of the gap  (spacing) 
$\delta_i:= \lambda_{i+1} -\lambda_i$, and more generally  the tail distribution of $\lambda_{i+l} -\lambda_i$, where $l$ is fixed and $ 1 \le i \le n-l$.  We are going to consider the following classes of random matrix ensembles.

{\it Wigner matrices.} A Wigner matrix $X_n$ is a random Hermitian matrices whose strictly upper triangular entries are iid sub-gaussian random variables with mean 0 and variance 1 and whose diagonal entries are independent sub-gaussian random variables with mean 0 and  variances bounded by $n^{1-o(1)}$, with the diagonal entries independent of the strictly upper triangular entries.  
Here and in the sequel we say that $\xi$ is a sub-gaussian random variable with sub-gaussian moment $m_0>0$ if one has $\P(|\xi|\ge t) \le m_0^{-1}\exp(-m_0t^2)$ for all $t>0$.


{\it Adjacency matrix of random graphs.} 
Let $G(n,p)$ be the Erd\H{o}s-R\'enyi graph on $n$ vertices with  edge density $p$. We denote by $A_n(p)$ the (zero-one) adjacency matrix of $G(n,p)$.

\vskip2mm

{\it Random matrix with arbitrary mean.} We consider 
a random Hermitian matrix $M_n$ of the form  $M_n:=F_n + X_n $, where 
 $F=F_n$ is a deterministic symmetric matrix of size $n$ and of norm $\|F_n\|_2=n^{O(1)}$, and  $X_n$ is a random Hermitian matrix where the off-diagonal 
 entries $\xi_{ij}, 1\le i<j\le n$ are iid copies of a random real variable $\xi$ of zero mean, unit variance, and finite $(2+ \eps_0)$-th moment with  fixed $\eps_0>0$.  
 The diagonal entries are independent random variables with mean zero and variance bounded by $O(1)$. 

 \vskip2mm

Here and later all asymptotic notations are used under the assumption that $n \rightarrow \infty$.

\vskip2mm

Gaps between consecutive eigenvalues have a central place in the theory of random matrices.  The limiting (global)  gap distribution for gaussian matrices (GUE and GOE) has been known for some time \cite{Mehbook}.   Recent progresses on the universality conjecture 
showed that these limiting distributions are universal with the class of Wigner matrices; see \cite{TVmehta, EYsurvey, TVsurvey1}  for surveys.  However, at the microscopic level, there are many open problems concerning basic questions. In the following discussion, when we say limiting distribution, we always mean after a proper normalization. 


The first natural question  is the limiting distribution of a given gap $\delta_i := \lambda_{i+1} -\lambda_i$.  For GUE, this was computed very recently by the second author \cite{Taogap}. Within the class of Wigner matrices, the  four moment theorem from \cite[Theorem 15]{TVuniversality} asserts that this distribution is universal, provided the four matching moment condition.  The matching moment condition was recently removed by Erd\H{o}s and Yau \cite{EY} using sophisticated techniques from the theory of parabolic PDE to analyze a Dyson Brownian motion, allowing for a computation of the gap distribution for random matrix ensembles such as the GOE or Bernoulli ensembles.

Another issue is to understand the size of the \emph{minimum} gap $\delta_{\operatorname{min}} := \min_{1 \le i \le n-1} (\delta_{i+1}-\delta_i)$. For the GUE ensemble, Bourgade and Ben-Arous  \cite{BB} showed  that the minimum gap $\delta_{\operatorname{min}}$  is of order $n^{-5/6} $ and computed the limiting distribution. 
To our best knowledge we are not aware of 
a polynomial lower bound  (of any fixed exponent) for $\delta_{min}$ for discrete random matrices, which are of  importance  
in applications in random graph theory and theoretical computer science. Even proving that $\delta_{\operatorname{min}} > 0$ (in other words the random matrix has simple spectrum) with high probability in the discrete case 
is already  a highly non-trivial problem,  first raised by Babai in the 1980s (motivated by his study of the isomorphism problem \cite{Babai1}). 
This latter problem was solved only very recently by the last two authors \cite{TVsimple}.

Our main goal is to provide lower tail bounds for the gaps $\delta_i$, and hence on the minimum gap $\delta_{\operatorname{min}}$. For the model $X_n$, by Wigner's  semi-circle law \cite{Mehbook},
most eigenvalues are in the interval $[-2 \sqrt n, 2 \sqrt n ]$, thus the average gap is of order $n^{-1/2} $. The question  is to estimate the 
 probability that a particular gap is  significantly smaller than the average:
 
 \vskip2mm 
 
 \begin{question} \label{main question} Estimate  $\P( \delta_i \le n^{-1/2} \delta )$, where $\delta$ may tend to zero with $n$.  \end{question} 
  
 \vskip2mm  As well known,  tail bounds (or deviation) are  essential
tools  in probability, and we  believe that good bounds for the probability in question will have a wide range of applications.

 Let us first discuss a few results related to this question. The last two authors showed \cite{TVuniversality} that for every constant $c_0>0$ there exists $c_1>0$ such that for Wigner matrices and for fixed $\eps>0$ one has

$$ \sup_{\eps n \leq i \leq (1-\eps) n} \P( \delta_i  \le n^{-c_0-\frac{1}{2}} ) \ll n^{-c_1}.$$
  The restriction to the bulk region $\eps n \leq i \leq (1-\eps) n$ was removed in \cite{TVedge}, and the mean zero requirement was partially removed in \cite{sean}.  
  The weakness of this theorem is that $c_1$ is small (much smaller than $1$, regardless the value of $c_0$), and thus one cannot use the union bound to conclude that $\delta _i >0$ simultaneously for all $i$.

In  \cite{ekyy}, Erd\H{o}s et. al. proved  for real Wigner matrices 
\begin{equation}\label{eck}
 \frac{1}{n} \sum_{\eps n \leq i \leq (1-\eps) n} \P(  \delta_i  \leq \delta n^{-1/2} ) \ll \delta^2,
\end{equation}
for any constant  $\eps,\delta>0$, with a similar result also available at the edge of the spectrum. The 
 quadratic decay $\delta^2$ here comes from an eigenvalue repulsion phenomenon, reflecting the first-order decay of the 
 two-point correlation function $\rho_2(x,y)$ of the GOE ensemble as one approaches the diagonal $x=y$.  However, this result only give a bound on the average probability, and furthermore 
 $\delta$ needs to be   a constant.  
  
Under some strong smoothness and decay hypotheses on the entries of a Hermitian Wigner matrix $X_n$, it was shown by Erd\H{o}s, Schlein, and Yau \cite{ESY} that one has the Wegner estimate
\begin{equation}\label{esy-est}
 \P\left( E n^{1/2} - \frac{\eps}{n^{1/2}} \leq \lambda_i \leq \lambda_{i+k-1} \leq E n^{1/2} + \frac{\eps}{n^{1/2}} \hbox{ for some } i \right) \ll \eps^{k^2} 
\end{equation}
for any fixed $k \geq 1$ and any $\eps>0$ and any bounded $E \in \R$ (note that the normalization in \cite{ESY} is different from the one used here).  Setting $k=2$ and applying the union bound (together with a standard bound on the operator norm of $X_n$), we conclude that
$$ \P( \delta_{\operatorname{min}} \leq \delta n^{-1/2} ) \ll n \delta^3 + \exp( - cn )$$
for some absolute constant $c>0$.  This is a strong (and essentially optimal) bound for small $\delta$, but it is only established for very smooth and rapidly decreasing complex probability distributions on the entries; in particular it does not apply at all to discrete ensembles such as the Bernoulli ensemble.  An analogue of \eqref{esy-est} for real smooth ensembles, with the exponent $k^2$ replaced by $k(k+1)/2$, was established in \cite[Appendix B]{BEYY}.
	
Finally, in \cite{TVsimple}, the last two authors established the bound
$$ \P( \delta_i = 0 ) \ll n^{-A} $$
for all $1 \leq i \leq n$ and any fixed $A>0$, for any of the three random matrix models (Wigner, Erd\H{o}s-Renyi, random matrix with arbitrary mean) considered above.  By the union bound (and shifting $A$ by $1$), this also gives
$$ \P( \delta_{\operatorname{min}} = 0 ) \ll n^{-A}.$$

 In this paper, we are going to give  answers to  Question \ref{main question}, and also  to the more general question of bounding 
 $\P( \lambda_{i+l} -\lambda_i  \le n^{-1/2} \delta )$, for a fixed $l$.  As with \cite{TVsimple}, our method is based on probabilistic  inverse Littlewood-Offord theorems (avoiding the use of moment comparison theorems or analysis of Dyson Brownian motion), and works  for any of the three random matrix ensembles introduced above, without requiring any smoothness on the distribution of the entries.  (However,  the two special models of Wigner matrices and adjacency matrices allow a more delicate analysis leading to better bounds.)   
 
 \section{Main results} 

For the sake of applications, we will be mainly focusing  on real symmetric matrices. All results can be extended to the complex case. 

We begin with the Wigner model, where our results are strongest.  Our first main theorem is as follows.

\begin{theorem}[Lower tail bound for a single gap]\label{theorem:main:Wigner} There is a constant $0< c <1$ (depending on the sub-gaussian moments) such that the following holds for the gaps $\delta_i := \lambda_{i+1}(X_n)-\lambda_i(X_n)$ of real symmetric Wigner matrices $X_n$. 
For any quantities $n^{-c} \le  \alpha \le c$  and  $ \delta\ge n^{-c/\alpha}$, we have

$$ \sup_{1 \leq i \leq n-1}  \P (\delta_i \le \delta n^{-\frac{1}{2}})  = O \left(  \frac{\delta}{ \sqrt {\alpha} } \right) .$$
 \end{theorem}

Notice that there is a trade-off between the quality of the bound and the range of $\delta$. At the one end, taking $\alpha$ to be a small constant, we have 

\begin{corollary} \label{bigend} 
For any constant $C_0 >0$ there is a constant $c_0 >0$ such that  for  real symmetric Wigner matrices $X_n$, we have

$$ \sup_{1 \leq i \leq n-1}  \P (\delta_i \le \delta n^{-\frac{1}{2}})   \le c_0 \delta, $$  
for all $\delta \ge n^{-C_0}$.  Consequently, by the union bound, 

$$\delta_{min}  \ge n^{-3/2 -o(1) } $$ with probability $1-o(1)$. 

\end{corollary} 

At the other end, taking $\alpha = n^{-c} $, we have 

\begin{corollary} 
Let $X_n$ be a  real symmetric Wigner matrix  where the upper triangular entries are Bernoulli (taking value $\pm 1$ with probability $1/2$), then 
$X_n$ has multiple eigenvalues with probability at most $O(\exp( - n^c))$ for some constant $c >0$. 
\end{corollary}

 This corollary improves  the polynomial bound  in \cite{TVsimple} for this specific case, although the result in \cite{TVsimple} holds in a more general context (the main assumption being that the distribution of any individual entry on the matrix is not concentrated at a single point).  We conjecture that the probability 
 that a Bernoulli matrix has  multiple eigenvalues is in fact $O( \exp(-cn))$ for some constant $c>0$.  

In the next theorem, we treat the gap $\lambda_{i+l} -\lambda_i$, where $l \geq 1$ is fixed.  Let  $d := \lfloor \log_2 l \rfloor$, and set 

$$c_l := \frac{ (3l +3 -2^{d+1} ) 2^d -1}{3}.$$

Thus $c_2=3$ and $c_l \ge \frac{l^2+2l}{3}$.

 
 \begin{theorem}[Repulsion for multiple gaps]\label{theorem:main:Wigner:d} There is a constant $0<c<l$ (depending on the sub-gaussian moments and on $l$) such that the following holds for real symmetric Wigner matrices $X_n$. 
For any quantities $n^{-c} \le   \alpha \le c$  and  $ \delta\ge n^{1 -c/\alpha}$

$$ \sup_{1 \leq i \leq n-l}  \P (|\lambda_{i+l}(X_n)-\lambda_i(X_n)| \le \delta n^{-\frac{1}{2}})  = O \left(  \left(\frac{\delta}{ \sqrt {\alpha} }  \right) ^{c_l }\right) .$$ 
\end{theorem}

 Similar to Corollary \ref{bigend}, we have 
 
 \begin{corollary} \label{bigend-d} 
For any constant $l, C_0 >0$ there is a constant $c_0 >0$ such that  for real Wigner matrices $X_n$

$$ \sup_{1 \leq i \leq n-l}  \P (|\lambda_{i+l}(X_n)-\lambda_i(X_n)| \le \delta n^{-\frac{1}{2}})   \le c_0 \delta^{c_l}, $$  for all $\delta \ge n^{-C_0}$.  
\end{corollary}

 The key feature 
 of this result is that 
 the bound  $\delta^{c_l} $  yields evidence for the  repulsion between nearby eigenvalues. For  $l=2$, we have $c_l=3$, yielding the bound  $O( \delta^3)$. If  there was no repulsion, 
 and the eigenvalues behave like point in a Poisson process, then the 
 bound would be  $O(\delta^2)$ instead.   The bound $c_l \ge \frac{l^2 +2l} {3} $ seems to be sharp, up to a constant factor; compare with \eqref{esy-est}.

We next consider the general model $M_n:= F_n + X_n$.
 It is often useful to   view $M_n$ as  a random perturbation of the  deterministic  matrix $F_n$, especially with respect to 
  applications in data science and numerical analysis, where matrices (as data or  inputs to algorithms) are often perturbed by  random noise.  One  can consult for instance \cite{TVcomp} where this viewpoint is discussed with respect to the least singular value problem.  As an illustration for this view point, we are going to present an application in  numerical analysis. 

\begin{theorem}\label{theorem:main:perturbation} Assume $\|F_n\|_2 \le n^{\gamma}$ for some  constant  $\gamma>0$. 
 Then for any fixed $A>0$, there exists $B>0$ depending on $\gamma,A$ such that 
$$ \sup_{1 \leq i \leq n-1}  \P(\delta_i \le n^{-B}) = O(n^{-A}),$$
where $\delta_i := \lambda_{i+1}(M_n)-\lambda_i(M_n)$ is the $i^{\operatorname{th}}$ gap of a random matrix $M_n = F_n +X_n$ with arbitrary mean.
In particular, the matrix $M_n$ has simple spectrum with probability $1-O(n^{-A})$ for any fixed $A>0$.
\end{theorem}

This theorem shows that  eigenvalue separation  holds  regardless the means of the entries. 
The dependence of $B$ in terms of $A$ and $\gamma$ in Theorem \ref{theorem:main:perturbation}  can be made explicit, for instance one can (rather generously) assume $B > (5A+6)\max \{1/2, \gamma\} +  5$.  

Finally, let us focus on the adjacency matrix of the random graph $G(n,p)$, where $0 < p< 1$ is a constant. 
Set  $F_n := p (J_n-I_n)$, where $J_n$ is the all-one matrix and $I_n$ is the identity. 
 In this case, we can sharpen the bound of Theorem \ref{theorem:main:perturbation} to almost match with the better bound in  Theorem \ref{theorem:main:Wigner}.

\begin{theorem}\label{theorem:main:ER}  Let $0 < p < 1$ be independent of $n$, and let $A_n$ be the adjacency matrix of the random graph $G(n,p)$.  Let $\delta_i := \lambda_{i+1}(A_n) - \lambda_i(A_n) $ denote the eigenvalue gaps.
Then for any fixed $A>0$, and any $\delta > n^{-A}$, we have
$$ \sup_{1 \leq i \leq n-1} \P( \delta_i \leq \delta n^{-\frac{1}{2}} ) =O(n^{o(1)} \delta).$$
\end{theorem}

All of our  results extend to the Hermitian case. In this case the upper triangular entries are complex variables 
 whose   real and complex components are iid copies of a sub-gaussian random variable of mean zero and variance 1/2. The value of  $c_l$ in Theorem \ref{theorem:main:Wigner:d} doubles;  
 see  Remark \ref{remark:Wigner:1:complex} for a discussion.

Our approach also works, with  minor modifications,  for random matrices  where the variance of the entries  decays  to zero with $n$.
 In particular, we can have them as small as  $n^{-1+c}$ for any fixed $c > 0$. 
 This case contains in particular the adjacency matrix of sparse random graphs.  Details will appear elsewhere.


\section{Applications}

\subsection{Random graphs have simple spectrum}

 Babai conjectured  that $G(n,1/2)$ has a simple spectrum, with probability $1-o(1)$. 
This conjecture was recently settled in \cite{TVsimple}.  Using the new deviation bounds, we can have the following 
stronger statement.

\begin{theorem}  With probability $1-o(1)$, the gap between any two eigenvalues of 
$G(n,1/2)$ is at least $n^{-3/2 +o(1) } $.
\end{theorem} 

To prove this theorem, apply  Theorem \ref{theorem:main:ER} with $\delta = n^{-1 -o(1) }$ and then apply the union bound.

\subsection{Nodal domains of random graphs}

Consider a random graph $G(n,p)$ (with $p$ constant) and the accompanying matrix $A_n (p)$. 
Let $u$ be an eigenvector of $A_n (p)$. A {\it strong}  nodal domain (with respect to $u$) 
 is a maximal connected component   $S$ of the graph, where for any $i, j \in S$, $u_i u_j >0$
 (here $u_i$ is the coordinate of $u$ corresponding to $i$).  A {\it weak} nodal domain 
 is defined similarly, but with $u_i u_j \ge 0$. Notice that strong nodal domains are disjoint, while the weak ones may overlap. 
 
 The notion of (weak and strong) nodal domains comes from Riemann geometry  and has become increasingly  useful in  graph theory and algorithmic applications
 (see  \cite{DLL, DLL24, DLL2, DLL5, DLL6}).  In \cite{DLL}, Dekel, Lee and Linial studied nodal domains of  the random graph $G(n,p)$ and raised the following conjectured  \cite[Question 1]{DLL}.

 \begin{conjecture}  With probability $1-o(1)$, all eigenvectors of $G(n,p)$ do not have zero coordinates. In other words, weak and strong nodal domains are the same. 
 \end{conjecture} 
 

We can now confirm this conjecture in the following stronger form

\begin{theorem}[Non-degeneration of eigenvectors]\label{theorem:perturbation:nodal}
Assume that $M_n=X_n+F_n$ as in Theorem \ref{theorem:main:perturbation}. Then for any $A$, there exists $B$ depending on $A$ and $\gamma$ such that
$$\P\Big(\exists \mbox{ an  eigenvector } v=(v_1,\dots,v_n) \mbox{ of $M_n$ with } |v_i|\le n^{-B} \mbox{ for some } i\Big) = O(n^{-A}). $$ 
\end{theorem}

 To make the picture complete, let us mention that recently Arora et. al. \cite{Arora}  proved that 
 with high probability, $G(n,p)$ has only two weak nodal domains, one corresponds to coordinates $u_i \ge 0$, and the other to 
 $u_i \le 0$. Combining this with Theorem \ref{theorem:perturbation:nodal}, we have 
 
 \begin{corollary}
 The following holds with probability $1-o(1)$ for $G(n,p)$. Each eigenvector has  exactly two strong nodal domains, which partition 
 the set of vertices. 
 \end{corollary} 
 
 \subsection{Numerical Analysis} 
 
Our  results can also be used to guarantee polynomial running time for certain algorithms. Let us consider an example. A basic   problem in numerical analysis  is to compute the leading eigenvector and eigenvalue of a large matrix.  A well-known power iteration method, designed for this purpose,  works as follows. Given $F$, a large symmetric matrix as input, let $u_0:= u$ be an arbitrary unit vector and consider 

$$u_{k+1} :=  \frac{ F u_k }{ \| F u_k \|_2 } . $$

For simplicity, assume that $F$ is positive semi-definite and let $0\le \lambda_1 \le \dots \le \lambda_n$ be the eigenvalues of $F$ (with, say, corresponding eigenvectors  $v_1, \dots, v_n$).  If $\lambda_n $ is strictly larger than $\lambda_{n-1}$, then $u_{k} $ converges to $v_n$ and $\| Fu_k \|_2$ converges to $\lambda_n$.

The heart of the matter is, of course, the rate of convergence, which is geometric with base $\frac{\lambda_{n-1} } {\lambda_n } $.  This means that to obtain 
an $\eps$ error term, we need to iterate $\Theta ( \frac{\lambda_n }{ \lambda_n -\lambda_{n-1} }  \log \frac{1}{\eps } )$ steps (see for instance \cite[Chapter 8]{GvL}.) 

For simplicity, assume that $\lambda_n = \| F \|_2  = n^{O(1)}$.  The algorithm is efficient (run in polynomial time)  if $\lambda_n -\lambda_{n-1} $ is polynomially large. However, if $\lambda_n -\lambda_{n-1} $ is exponentially small in $n$, then the algorithm  take exponentially many steps. 

One can use our result to avoid this difficulty.  The idea is to  artificially perturb $F$ by a random matrix $ X_n$.
Theorem \ref{theorem:main:perturbation} shows that with high probability, the gap $\lambda_n (F +  X_n ) - \lambda_{n-1} (F + X_n)$ is polynomially large, 
which in turn guarantees a polynomial running time. On the other hand, adding $X_n$ (or a properly scaled version of it) will not change 
$\lambda_n$  and $v_n$ significantly. For instance, one  can immediately apply Weyl's bound here, but 
 better results  are available in more specific cases. 
 
This argument is closely related to the notion of {\it smoothed analysis} introduced by Spielman and Teng \cite{ST}. We will discuss similar applications in a separate note.

\section{Proof strategy}\label{subsection:strategy} 

We first  consider  Theorem \ref{theorem:main:Wigner}.  In what follows, all eigenvectors have unit length. Assume that the entries $\xi_{ij}$ are iid copies of a random variable $\xi$. 
Consider a fixed vector $x =(x_1, \dots, x_n)$, we  introduce the \emph{small ball probability}  

$$\rho_\delta(x):=\sup_{a\in \R}\P(|\xi_1 x_1+\dots + \xi_n x_n-a|\le \delta).$$

where $\xi_1,\dots,\xi_{n}$ are iid  copies of $\xi$. Let $X_{n-1} $ be an $(n-1) \times (n-1)$ minor of $X_n$. 
In this section we will to reduce Theorem \ref{theorem:main:Wigner} to the following

\begin{theorem}\label{theorem:Wigner:poor}   There exist positive constants $c$ and $\alpha_0$ such that the following holds with probability $1-O(\exp(-\alpha_0n))$: for all $n^{-c}<\alpha <c$ and 
every unit eigenvector $v$ of $X_{n-1}$  and for all $\delta \ge n^{-c/\alpha}$,

$$ \rho_{\delta}(v) =O \left(\frac{\delta}{\sqrt{\alpha}}\right).$$
\end{theorem}

We will discuss how to establish this result later in this section, and give the formal details of proof in later sections.
Assuming this theorem for the moment, we now finish the proof of Theorem \ref{theorem:main:Wigner}.
To start, let  $\CB$   be the event that the spectrum of $X_n$ belongs to a controlled interval,

\begin{equation}\label{eqn:Xn:norm}
\CB=\{\omega: \lambda_1(X_n),\dots,\lambda_n(X_n) \subset [-10\sqrt{n},10 \sqrt{n}]\}.
\end{equation}

Standard results in random matrix theory (such as \cite{AGZ},\cite{Tao-book}) shows that $\CB$ holds with probability  $1- \exp(-\Theta(n))$. 


Now let $1\le i\le n-1$, and consider the event $\CE_i$ that $X_n $  satisfies  $\lambda_{i+1}-\lambda_i \le \delta n^{-\frac{1}{2}}$.
For each $1\le j\le n$, let $\CG_{i,j}$ be the event $\CE_i$ and that  the  eigenvector $w=(w_1,\dots,w_n)^T$ with eigenvalue $\lambda_i(M_{n})$ satisfies
$|w_j|\gg  T$, where $T$ is to be chosen later. The following delocalization result helps us to choose $T$ properly.

\begin{theorem}\label{theorem:delocalization}   
 Let $X_n$ be a Wigner matrix as in Theorem \ref{theorem:main:Wigner}. There are constants $c_1 c_2, c_3  >0$ such that the following holds. 
 With probability at least $1 -\exp( -c_1 n)$, every eigenvector of $X_n$ has at least $c_2 n$ coordinates with absolute value at least $c_3 n^{-1/2} $. The same statement holds for the adjacency matrix $A_n (p)$, where 
 $p$ is a constant. (The constants $c_1,c_2, c_3$ may depend on the distribution of the entries of $X_n$ and on $p$ in the $A_n (p)$ case.) 
\end{theorem} 

We prove this theorem in Appendix \ref{localization}.

Write
\begin{equation}\label{discussion:mn-split}
 X_n = \begin{pmatrix} X_{n-1} & X \\ X^* & x_{nn} \end{pmatrix}
\end{equation}

where $X$ is a column vector.  From the Cauchy interlacing law, we observe that $\lambda_i(X_n) \le \lambda_i(X_{n-1})\le \lambda_{i+1}(X_n)$. Let 
$u$ be  the (unit) eigenvector of $\lambda_i (X_n) $; we write 
$u= (w,b)$, where $w$ is a vector of length $n-1$ and $b$ is a scalar.  We have 

$$ \begin{pmatrix} X_{n-1} & X \\ X^* & x_{nn} \end{pmatrix} \begin{pmatrix} w \\ b \end{pmatrix} = \lambda_i(X_n) \begin{pmatrix} w \\ b \end{pmatrix} .$$

Extracting the top $n-1$ components of this equation we obtain

$$ (X_{n-1}-\lambda_i(X_n)) w + b X = 0.$$

Let  $v $ be the unit eigenvector  of $X_{n-1}$ corresponding to  $\lambda_i(X_{n-1})$. By multiplying with $v^T$, we obtain 

$$|b v^T X | = |v^T  (X_{n-1}-\lambda_i(X_n)) w| =|\lambda_i(X_{n-1})-\lambda_i(X_n)| |v^T w| .$$

We conclude that, if $\CE_i$ holds, then  $|b v^T X| \le \delta n^{-1/2} $. If we assume that $\CG_{i,n}$ also holds, then we therefore have

$$ |v^T  X | \le \delta \frac{n^{-1/2} }{T} . $$

If one can choose $T = \Omega (n^{-1/2} )$, the RHS is $O(\delta)$, and thus we reduce the problem to bounding  the probability that $X_{n-1}$ has an eigenvector 
$v$ such that $|v^T X| = O(\delta)$, for which we can apply  Theorem \ref{theorem:Wigner:poor}.

Of course we cannot assume that the last coordinate $b$ of the eigenvector of $\lambda_i (X_n)$ to be large. Apparently, this eigenvector has a coordinate 
of order $n^{-1/2} $, but a trivial union bound argument would cost us a factor $n$. We can avoid this factor by using Theorem \ref{theorem:delocalization}. 

Notice that there is no specific reason to look at the last coordinate. Thus, if we instead look at a random coordinate (uniformly between $1$ and $n$ and split $X_n$ accordingly), then 
we have 

$$\P (\CE_i) \le \frac{n}{N} \P \left( |v^T X | \le \delta \frac{n^{-1/2} }{T} \right) + \P( n_T < N),  $$  where $n_T$ is the number of coordinates with absolute values at least $T$; and $v,X$ are the corresponding eigenvector and column vector.

Now choose $T = c_3 n^{-1/2} $ and $N= c_2 n$.  By Theorem \ref{theorem:delocalization}, $\P (n_T < N) \le \exp( -c_1 n)$. Thus, we have 

\begin{equation}\label{badevent1} 
\P (\CE_i) \le  c_2^{-1} \P \left( |v^T X | \le c_3^{-1} \delta  \right) + \exp(- c_1 n).  
\end{equation} 

To bound the RHS, we recall the definition of small probability at the beginning of the section. After a proper rescaling of $\delta$ by a factor $c_3^{-1}$, for any $\eps >0$

$$ \P(| v^T X | \le   \delta )   \le \P (|v^T X | \le \delta |  \rho_{\delta} (v)  < \eps) + \P (\rho_{\delta} (v) \ge \eps)  \le \eps + \P( \rho_{\delta}(v) \ge \eps). $$

Set $\eps := C \frac{\delta}{\alpha }$ for a sufficiently large constant $C$, Theorem \ref{theorem:Wigner:poor} then implies that 

\begin{equation} \label{badevent2}  
\P(| v^T X | \le  \delta )   \le C \frac{\delta}{\sqrt \alpha} + \exp(- \Omega (n) ). 
\end{equation} 

This completes the proof of Theorem \ref{theorem:main:Wigner} assuming Theorem \ref{theorem:Wigner:poor}. To prove Theorem  \ref{theorem:main:Wigner:d}, we will follow the same strategy, but the analysis is more delicate. We refer the reader to Section \ref{section:Wigner:d} for more detail.

For treating the general model $M_n := F_n + X_n$, instead of 
Theorem \ref{theorem:Wigner:poor}, a similar argument reduces one to the task of proving the following analogue of Theorem \ref{theorem:Wigner:poor}.

\begin{theorem}\label{theorem:perturbation:poor}  For any $A,\gamma>0$, there exist $\alpha_0,B>0$ depending on $A$ and $\gamma$ such that the following holds with probability $1-O(\exp(-\alpha_0n))$: every unit eigenvector $v$ of $M_{n-1}=X_{n-1}+F_{n-1}$ with eigenvalue $\lambda$ of order $n^{O(1)}$ obeys the anti-concentration estimate
$$ \rho_{n^{-B}}(v) = O(n^{-A}).$$
\end{theorem}

Let us now discuss  the proof of Theorem  \ref{theorem:Wigner:poor} and Theorem \ref{theorem:perturbation:poor}.  This proof relies on the general theory of small ball probability. 
Recent 
results from inverse Littlewood-Offord theory developed by the second and third author (see e.g. \cite{TVsurvey}, \cite{NVsur}) and by Rudelson and Vershynin (see e.g. \cite{rv}) show that if 
the vector $x$ does not have rich {\it additive structure}, then the small ball probability $\rho_\delta(x)$ is close to $\delta$.
Thus,  the key ingredient for proving Theorem \ref{theorem:Wigner:poor} and Theorem \ref{theorem:perturbation:poor} is to quantify the following

\vskip .2in
\begin{heuristic}\label{heuristic}
With very high probability  the  eigenvectors of $M_{n-1}$ do not have additive structures.
\end{heuristic}


\subsection{Notation}\label{notation-sec}

Throughout this paper, we regard $n$ as an asymptotic parameter going to infinity (in particular, we will implicitly assume that $n$ is larger than any fixed constant, as our claims are all trivial for fixed $n$), and allow all mathematical objects in the paper to implicitly depend on $n$ unless they are explicitly declared to be ``fixed'' or ``constant''.  We write $X = O(Y)$, $X \ll Y$, or $Y \gg X$ to denote the claim that $|X| \leq CY$ for some fixed $C$; this fixed quantity $C$ is allowed to depend on other fixed quantities such as the $(2+\eps_0)$-moment (or the sub-gaussian parameters) of $\xi$  unless explicitly declared otherwise.  We also use $o(Y)$ to denote any quantity bounded in magnitude by $c(n) Y$ for some $c(n)$ that goes to zero as $n \to \infty$.  Again, the function $c(.)$ is permitted to depend on fixed quantities. 

For a square matrix $M_n$ and a number $\lambda$, for short we will write $M_n-\lambda$ instead of $M_n-\lambda I_n$. All the norms in this note, if not specified, will be the usual $\ell_2$-norm.

The rest of this paper is organized as follows. We first give a full proof of Theorem \ref{theorem:main:Wigner} and  Theorem \ref{theorem:main:Wigner:d} in Section \ref{section:Wigner:1} and Section \ref{section:Wigner:d}. We then prove Theorem \ref{theorem:main:perturbation} in Section \ref{section:perturbation}, and  sharpen the estimates for Erd\H{o}s-R\'enyi graphs in Section \ref{section:ER}. The note is concluded by an application in Section \ref{section:nodal}.

\section{Consecutive gaps for Wigner matrices: proof of Theorem \ref{theorem:Wigner:poor}}\label{section:Wigner:1}

In this section we prove Theorem \ref{theorem:Wigner:poor}, which as discussed previously implies Theorem \ref{theorem:main:Wigner}.
Our treatment in this section is based on the work by Vershynin in \cite{vershynin}. First of all, we recall the definition of compressible and incompressible vectors.

\begin{definition} Let $c_0, c_1 \in (0,1)$ be two numbers (chosen depending on the sub-gaussian moment of $\xi$.) A vector $x \in \R^n$ is called {\it sparse} if $|\supp(x)| \le c_0 n$.
A vector $x \in S^{n-1}$ is called {\it compressible} if $x$ is within Euclidean distance $c_1$ from the set of all sparse vectors.
  A vector $x \in S^{n-1}$ is called {\it incompressible} if it is not compressible.
  
  The sets of compressible and incompressible vectors in $S^{n-1}$
  will be denoted by $\Comp(c_0, c_1)$ and $\Incomp(c_0, c_1)$ respectively.
\end{definition}

\subsection{The compressible case}\label{subsection:Wigner:compressible} Regarding the behavior of $X_n x$ for compressible vectors, the following was proved in \cite{vershynin}. 
\begin{lemma}\cite[Proposition 4.2]{vershynin}\label{prop:comp}
There exist positive constants $c_0,c_1$ and $\alpha_0$ such that the following holds for any $\lambda_0$ of order $O(\sqrt{n})$. For any fixed $u\in \R^n$ one has 

$$\P(\inf_{x \in \Comp(c_0,c_1)} \|(X_n-\lambda_0)x-u\| \ll \sqrt{n}) = O(\exp(-\alpha_0 n )).$$
\end{lemma}

We also refer the reader to Subsection \ref{subsection:perturbation:compressible} for a detailed proof of a similar statement. As a corollary, we deduce that eigenvectors are not compressible with extremely high probability.

\begin{lemma}\label{lemma:Wigner:comp} There exist positive constants $c_0,c_1$ and $\alpha_0$ such that 

$$\P\left(\exists \mbox{ a unit eigenvector } v \in \Comp(c_0,c_1)\right)=O( \exp(-\alpha_0 n)).$$
\end{lemma}

\begin{proof}[Proof of Lemma \ref{lemma:Wigner:comp}] Assuming \eqref{eqn:Xn:norm}, we can find $\lambda_0$ as a multiple of $n^{-2}$ inside $[-10 \sqrt{n}, 10 \sqrt{n}]$ such that $|\lambda -\lambda_0| \le n^{-2}$. Hence

$$\|(X_n-\lambda_0) v\| = \|(\lambda-\lambda_0) v\|\le n^{-2}.$$ 

On the other hand, for each fixed $\lambda_0$, by Lemma \ref{prop:comp},  

$$\P(\exists v\in  \Comp(c_0,c_1): \|(X_n-\lambda_0)v\| \le n^{-2}) = O(\exp(-\alpha_0 n )).$$ 

The claim follows by a union bound with respect to $\lambda_0$. 
\end{proof}

\subsection{The incompressible case}\label{subsection:Wigner:incompressible}  Our next focus is on incompressible vectors. This is the treatment where the inverse Littlewood-Offord ideas come into play. 

We first introduce the notion of {\it least common denominator} by Rudelson and Versynin (see \cite{rv}). Fix parameters $\kappa$  and $\gamma$ (which may depend on $n$), where $\gamma \in (0,1)$. For any nonzero vector $x$ define
$$
\LCD_{\kappa,\gamma}(x)
:= \inf \Big\{ \theta> 0: \dist(\theta x, \Z^n) < \min (\gamma \| \theta x \|,\kappa) \Big\}.
$$

\begin{remark}\label{remark:kappa&gamma}
In application, we will choose $\kappa=n^{\kappa_0}$ for some sufficiently small $\kappa_0$. Regarding the parameter $\gamma$, it suffices to set it to be $1/2$ for this section; but we will choose it to be proportional to $\alpha$ (from Theorem \ref{theorem:main:Wigner}) in the next section, hence $\gamma$ can decrease to zero together with $n$. The requirement that the distance is smaller than $\gamma\|\theta x \|$ forces us to consider only non-trivial integer points as approximations of $\theta x$. The inequality $\dist(\theta x, \Z^n) < \kappa$ then yields that most coordinates of $\theta x$ are within a small distance from  non-zero integers.
\end{remark}

\vskip .2mm

\begin{theorem}[Small ball probability via LCD]\cite{rv}\label{theorem:RV:1}  Let  $\xi$ be a sub-gaussian random variable of mean zero and variance one, and let $\xi_1, \ldots, \xi_n$ be iid copies of $\xi$. Consider a vector $x\in \R^n$ which satisfies
$\|x\| \ge 1$. Then, for every $\kappa > 0$ and $\gamma \in (0,1)$, and for
  $$
  \eps \ge \frac{1}{\LCD_{\kappa,\gamma}(x)},
  $$
  we have
  $$
  \rho_\eps(x) =O\left(\frac{\eps}{\gamma} + e^{-\Theta(\kappa^2)}\right),
  $$

where the implied constants depend on $\xi$.
\end{theorem}

To deal with symmetric or Hermitian matrices, it is more convenient to work with the so-called  {\em regularized least common denominator}. Let $x=(x_1,\dots,x_n)\in S^{n-1}$ . Let $c_0,c_1\in (0,1)$ be given constants, and assume $x\in \Incomp(c_0,c_1)$. It is not hard to see that there are at least $c_0c_1^2n/2$ coordinates $x_k$ of $x$ satisfy

\begin{equation}\label{eqn:x_k}
\frac{c_1}{\sqrt{2n}}\le |x_{k}| \le \frac{1}{\sqrt{ c_0n}}.
\end{equation}

Thus for every $x\in \Incomp(c_0,c_1)$ we can assign a subset $\spread(x)\subset [n]$ such that \eqref{eqn:x_k} holds for all $k\in \spread(x)$ and

$$|\spread(x)| = \lceil c' n\rceil,$$

where we set

\begin{equation}\label{eqn:c'}
c':=c_0c_1^2/4.
\end{equation}

\begin{definition}[Regularized LCD, see also \cite{vershynin}]\label{def reg LCD}
  Let $\alpha \in (0, c'/4)$. 
  We define the {\em regularized LCD} of a vector $x \in \Incomp(c_0,c_1)$ as
  $$
  \LCDhat_{\kappa,\gamma}(x,\alpha) = \max \Big\{ \LCD_{\kappa,\gamma} \big(x_I/\|x_I\|\big) : \, I \subseteq \spread(x), \, |I| = \lceil \alpha n \rceil \Big\}.
  $$
\end{definition}

Roughly speaking, the reason we choose to work with $\LCDhat$ is that we want to detect structure of $x$ in sufficiently small segments. From the definition, it is clear that if $\LCD(x)$ is small (i.e. when $x$ has strong structure), then so is $\LCDhat(x,\alpha)$.

\begin{lemma}\label{lemma:comparison:regularized} For any $x\in S^{n-1}$ and any $0<\gamma < c_1\sqrt{\alpha}/2$, we have 
$$\LCDhat_{\kappa, \gamma (c_1\sqrt{\alpha}/2)^{-1}}(x,\alpha) \le \frac{1}{c_0}\sqrt{\alpha}\LCD_{\kappa,\gamma}(x).$$

Consequently, for any $0<\gamma <1$

$$\LCDhat_{\kappa, \gamma}(x,\alpha) \le \frac{1}{c_0}\sqrt{\alpha}\LCD_{\kappa,\gamma (c_1\sqrt{\alpha}/2)}(x).$$

\end{lemma}

\begin{proof}[Proof of Lemma \ref{lemma:comparison:regularized}] Note that for any $I\subset \spread(x)$ with $|I|=\lceil \alpha n\rceil $, 

$$\frac{c_1}{2}\sqrt{|I|/n} < \|x_I\| \le \frac{1}{\sqrt{c_0}} \sqrt{|I|/n} .$$

Assume that  $\dist(tx, \Z^n) <  \min(\gamma\|tx \|_2,\kappa)$ for some $t\approx \LCD_{\kappa,\gamma}(x)$. Define 

$$t_I:= t\|x_I\|.$$ 

One then has 

$$ \frac{c_1}{2} \sqrt{\alpha} t <   t_I \le \frac{1}{\sqrt{c_0}} \sqrt{|I|/n} t .$$

Furthermore, 
\begin{align*}
\dist (t_I x_I/\|x_I\|, \Z^I) &\le \dist(t x, \Z^n) \\
&<  \min(\gamma\|tx \|,\kappa) \\
& \le \min (\gamma (\frac{c_1}{2}\sqrt{\alpha} )^{-1} \|t_I (x_I / \|x_I\|) \|,\kappa).
\end{align*}

Thus 

$$\LCD_{\kappa,\gamma (c_1 \sqrt{\alpha}/2)^{-1}}(x_I/\|x_I\|) \le  t_I \le  \frac{1}{\sqrt{c_0}} \sqrt{|I|/n}  t \le \frac{1}{c_0}\sqrt{\alpha} \LCD_{\kappa,\gamma}(x).$$
\end{proof}

We now introduce a result connecting the small ball probability with the regularized LCD.

\begin{lemma}\label{lemma:smallball:regularized} Assume that 

$$\eps\ge  \frac{1}{c_0} \sqrt{\alpha}  (\LCDhat_{\kappa,\gamma}(x,\alpha))^{-1}.$$ 

Then we have 
$$\rho_\eps(x) =  O\left(\frac{ \eps}{\gamma  c_1\sqrt{\alpha}} + e^{-\Theta(\kappa^2)}\right).$$

\end{lemma}

\begin{proof}[Proof of Lemma \ref{lemma:smallball:regularized}] First apply Theorem \ref{theorem:RV:1} to $x_I/\|x_I\|$ where the $\LCDhat$ is achieved: for any $\delta \ge (\LCD_{\kappa,\gamma}(x_I/\|x_I\|))^{-1}$

$$
\rho_\delta\left(\frac{1}{\|x_I\|}S_I\right) =O\left(\frac{\delta}{\gamma} + e^{-\Theta(\kappa^2)}\right),
$$

where $S_I = \sum_{i\in I} \xi_i x_i$.

Recall that  $\frac{c_1}{2}\sqrt{\alpha} < \|x_I\| \le \frac{1}{c_0} \sqrt{\alpha}$. Also notice that if $ I \subset J \subset [n]$, then 

$$\rho_\delta(S_J)\le \rho_\delta(S_I).$$

Thus for any $\eps \ge  \frac{1}{c_0} \sqrt{\alpha} (\LCD_{\kappa,\gamma}(x_I/\|x_I\|))^{-1}$, one has $\eps/\|x_I\| \ge (\LCD_{\kappa,\gamma}(x_I/\|x_I\|))^{-1}$, and so

$$
\rho_\eps(x)\le \rho_\eps(S_I) = \rho_{\eps /\|x_I\|}(S_I/\|x_I\|) \ll \frac{\eps}{\gamma \|x_I\|} + e^{-\Theta(\kappa^2)} \ll \frac{ \eps}{\gamma  c_1\sqrt{\alpha}} + e^{-\Theta(\kappa^2)}.
$$

\end{proof}

For given $D,\kappa,\gamma$ and $\alpha$, we denote the set  of unit vectors with bounded regularized LCD by

$$T_{D,\kappa,\gamma,\alpha}:=\{x\in \Incomp(c_0,c_1):  \LCDhat_{\kappa,\gamma}(x,\alpha) \le D \}.$$

The following is an analog of \cite[Lemma 7.9]{vershynin}.

\begin{lemma}\label{lemma:Wigner:key}
There exist $c>0,\alpha_0>0$ depending on $c_0,c_1$ from Lemma \ref{lemma:Wigner:comp} such that the following holds with $\kappa=n^{2c}$ and $\gamma=1/2$. Let $n^{-c}\le \alpha \le c'/4$, and $1\le D \le n^{c/\alpha}$. Then for any fixed $u\in \R^n$ and any real number $\lambda_0$ of order $O(\sqrt{n})$,

$$\P\left(\exists x\in T_{D,\kappa,\gamma,\alpha} : \|(X_n-\lambda_0) x -u\| =o (\beta \sqrt{n})\right) = O(\exp(-\alpha_0 n)),$$

where

$$\beta :=\frac{\kappa}{\sqrt{\alpha} D}.$$

\end{lemma}

We will give a proof of Lemma \ref{lemma:Wigner:key} in Appendix \ref{section:Wigner:key} by following \cite{vershynin}. Assuming it for now, we will obtain the following key lemma.

\begin{lemma}\label{lemma:Wigner:incomp} With the same assumption as in Lemma \ref{lemma:Wigner:key}, we have
$$\P\left(\exists \mbox{ a unit eigenvector } v \in \Incomp(c_0,c_1): \LCDhat_{\kappa,\gamma}(v,\alpha)\le n^{c/\alpha}\right) =O( \exp(-\alpha_0 n)).$$
\end{lemma}

\begin{proof}[Proof of Lemma \ref{lemma:Wigner:incomp}]
Set $D =n^{c/\alpha}$ and $\beta = \frac{\kappa}{\sqrt{\alpha} D}$. Assuming \eqref{eqn:Xn:norm}, we first approximate $\lambda$ by  a multiple of $\beta$, called $\lambda_0$, from the interval  $[-10 \sqrt{n}, 10 \sqrt{n}]$.  

As $X_n v = \lambda v$, we have

\begin{equation}\label{eqn:bound:beta}
\|(X_n-\lambda_0)v\|=\|(\lambda-\lambda_0)v\| =O(\beta)=o(\beta \sqrt{n}).
\end{equation}

On the other hand, by Lemma \ref{lemma:Wigner:key}, \eqref{eqn:bound:beta} holds with probability $O(\exp(-\alpha_0 n))$ for any fixed $\lambda_0$. Taking a union bound over the $O(\beta^{-1}\sqrt{n})$ choices of $\lambda_0$, one obtains

\begin{align*}
&\P\left(\exists v\in T_{D,\kappa,\gamma,\alpha}, \exists \lambda_0 \in \beta \Z \cap [-10\sqrt{n},10\sqrt{n}] : \|(X_n-\lambda_0)v\| =o( \beta \sqrt{n})\right)\\ &\ll \exp(-\alpha_0 n) \times \beta^{-1}\sqrt{n}\\
&\ll \exp(-\alpha_0 n/2).
\end{align*}

\end{proof}

Putting Lemma \ref{lemma:Wigner:comp} and Lemma \ref{lemma:Wigner:incomp} together, we obtain a realization of Heuristic \ref{heuristic} as follows.

\begin{theorem}\label{theorem:Wigner:1:together} There exists a positive constants $c$ and $\alpha_0$ depending on $c_0,c_1$ from Lemma \ref{lemma:Wigner:comp} such that for  $\kappa=n^{2c}, \gamma=1/2$ and for $n^{-c}\le \alpha \le c'/4$,
 
$$\P\left(\mbox{All unit eigenvectors } v \mbox{ belong to } \Incomp(c_0,c_1) \mbox{ and } \LCDhat_{\kappa,\gamma}(v,\alpha)\gg n^{c/\alpha}\right)=1-O( \exp(-\alpha_0 n/2)).$$
\end{theorem}

We can now complete the proof of  Theorem \ref{theorem:Wigner:poor}. By the theorem above, it is safe to assume $\LCDhat_{\kappa,\gamma}(v,\alpha)\gg n^{c/\alpha}$ for all eigenvectors of $X_{n-1}$. As such, by Lemma \ref{lemma:smallball:regularized}, for any $\delta\gg \sqrt{\alpha}/n^{c/\alpha}$ we have 

\begin{equation}\label{eqn:Wigner:1:ball}
\rho_\delta(S) \ll \frac{\delta}{\sqrt{\alpha}} + e^{-\Theta(\kappa^2)} \ll \frac{\delta}{\sqrt{\alpha}},
\end{equation}

where we recall that $\kappa= n^{2c}$.

Theorem \ref{theorem:main:Wigner} then follows from \eqref{badevent1} and Theorem \ref{theorem:Wigner:poor}, where we note that the condition $\delta\ge n^{-c/\alpha}$ automatically implies $\delta\gg \sqrt{\alpha}n^{-c/\alpha}$.

\vskip .2mm

\begin{remark}
We could have proved Theorem \ref{theorem:main:Wigner} by directly applying the results from \cite{vershynin}, where $\LCD$ and $\LCDhat$ were defined slightly differently. However, the current forms of $\LCD$ and $\LCDhat$ are easy to extend to higher dimensions, which will be useful for our proof of Theorem \ref{theorem:main:Wigner:d} next. 
\end{remark}

\section{Proof of Theorem \ref{theorem:main:Wigner:d}}\label{section:Wigner:d}

Let us first give a full treatment for the case $l=2$. Here $c_l=3$ and we are considering $\lambda_{i+2} -\lambda_i$.
The case of general $l$ case will be deduced with some minor modifications. 

\subsection{Treatment for $l=2$} First of all,  we will introduce the extension of LCD to higher dimension, following Rudelson and Vershynin (see \cite{rv-rec}). Consider two unit vectors $x_1=(x_{11},\dots,x_{1n}), x_2=(x_{21},\dots, x_{2n}) \in \R^n$. Although the following definition can be extended to more general $x_1,x_2$, let us assume them to be orthogonal. Let $H_{x_1,x_2}\subset \R^n$ be the subspace generated by $x_1,x_2$. Then for $\kappa > 0$ and $\gamma \in (0,1)$ we define,

$$
\LCD_{\kappa,\gamma}(x_1, x_2)
:= \inf_{x\in H_{x_1,x_2}, \|x\|=1} \LCD_{\kappa,\gamma}(x).
$$

Similarly to Theorem \ref{theorem:RV:1}, the following result gives a bound on the small ball probability for the $\R^2$-random sum $S=\sum_{i=1}^n \ep_i (x_{i1},x_{i2})$ in terms of the joint structure of $x_1$ and $x_2$.

\begin{theorem}\cite{rv-rec}\label{theorem:RV:d}
 Let  $\xi$ be a sub-gaussian random variable of mean zero and variance one, and let $\xi_1, \ldots, \xi_n$ be iid copies of $\xi$. Then, for every $\kappa > 0$ and $\gamma \in (0,1)$, and for
  $$
  \eps \ge \frac{1}{\LCD_{\kappa,\gamma}(x_1, x_2)},
  $$
  we have
  $$
  \rho_{\eps}(S) \ll \left( \frac{\eps}{\gamma}\right)^2 +e^{-\Theta(\kappa^2)},
  $$
  where the implied constants depend on $\xi$.
\end{theorem}

This theorem plays a crucial role in our task of obtaining the repulsion. 

We are now ready to present the main idea of the proof of Theorem \ref{theorem:main:Wigner:d}. Let $1\le i\le n-2$, and consider the event $\CE_i$ that $X_n$ has three eigenvalues $\lambda_i(X_n)\le  \lambda_{i+1}(X_n) \le \lambda_{i+2}(X_n)$ with $\lambda_{i+2}-\lambda_i \le \delta/n^{1/2}$ for some $\delta$.  For short, we set

$$t:=\delta/n^{1/2}.$$

Write

\begin{equation}\label{discussion:mn-split:d}
 X_n = \begin{pmatrix} X_{n-1} & X \\ X^* & m_{nn} \end{pmatrix}
\end{equation}

for an $(n-1) \times (n-1)$ minor $X_{n-1}$ and a column vector $X$.  

From the Cauchy interlacing law, we observe that 

$$\lambda_i(X_n) \le \lambda_i(X_{n-1})\le \lambda_{i+1}(X_n) \le \lambda_{i+1}(X_{n-1})\le \lambda_{i+2}(X_n).$$ 

By definition, for some eigenvector $(w,b)$ with $|b|\gg n^{-1/2}$

$$ \begin{pmatrix} X_{n-1} & X \\ X^* & m_{nn} \end{pmatrix} \begin{pmatrix} w \\ b \end{pmatrix} = \lambda_i(X_n) \begin{pmatrix} w \\ b \end{pmatrix}.$$

Extracting the top $n-1$ components of this equation we obtain

$$ (X_{n-1}-\lambda_i(X_n)) w + b X = 0.$$

If we choose unit eigenvectors $v =(v_1,\dots,v_{n-1})^T$ of $X_{n-1}$ with eigenvalue $\lambda_i(X_{n-1})$, and $v' =(v_1',\dots,v_{n-1}')^T$ of $X_{n-1}$ with eigenvalue $\lambda_{i+1}(X_{n-1})$ respectively, then

$$|b v^T X | = |v^T  (X_{n-1}-\lambda_i(X_n)) w| =|\lambda_i(X_{n-1})-\lambda_i(X_n)| |v^T w| \le t;$$

and similarly

$$|b {v'}^T X | = |{v'}^T  (X_{n-1}-\lambda_i(X_n)) w| =|\lambda_{i+1}(M_{n-1})-\lambda_i(X_n)| |{v'}^T w| \le t.$$



In summary, we have

\begin{equation}\label{eqn:Wigner:d:F}
|v^T X |  =O(\delta)  \wedge |{v'}^T X |  = O(\delta).
\end{equation}

Let $\CF_i$ denote this event. By Theorem \ref{theorem:RV:d}, if we can show $\LCD_{\kappa,\gamma}(v,v')$ large, then the $\P(\CF_i)$ can be estimated quite efficiently. In what follows we will focus on $\inf_{u\in H_{v,v'}} \LCD_{\kappa,\gamma}(u)$ by first studying $\inf_{u\in H_{v,v'}} \LCDhat_{\kappa,\gamma}(u,\alpha)$, and then using Lemma \ref{lemma:comparison:regularized} to pass back to $\LCD$.

We start with a simple fact first.
\begin{fact}\label{fact:Wigner:d} For any unit vector $u$ from $H_{v,v'}$, we have 

$$\|(X_{n-1}-\lambda_i(X_{n-1}))^2 u\|\le t^2.$$

In particular, by the Cauchy-Schwarz inequality

\begin{equation}\label{eqn:Wigner:d:CS}
\|(X_{n-1}-\lambda_i(X_{n-1}))u\|\le t.
\end{equation}
\end{fact}

\begin{proof}[Proof of Fact \ref{fact:Wigner:d}] Assume that $u=av+a'v'$ with $a^2+{a'}^2=1$, then 

\begin{align*}
\|(X_{n-1}-\lambda_i(X_{n-1}))^2 u\| &=\|(X_{n-1}-\lambda_i(X_{n-1}))(X_{n-1}-\lambda_i(X_{n-1}))u \|\\
&= \|(X_{n-1}-\lambda_i(X_{n-1}))(X_{n-1}-\lambda_i(X_{n-1}))v'\|\\
&=|a' (\lambda_{i+1}(X_{n-1})-\lambda_i(X_{n-1}))^2| \\
&\le t^2.
\end{align*}

\end{proof}

Next, from Lemma \ref{lemma:Wigner:comp} and Lemma \ref{lemma:Wigner:key}, and by \eqref{eqn:Wigner:d:CS} from Fact \ref{fact:Wigner:d}, we infer the following.

\begin{theorem}\label{theorem:Wigner:d:poor} There exist positive constants $c,\alpha_0$ such that for $\kappa=n^{2c}, \gamma=1/2$, and for any $n^{-c}\le \alpha \le c'/4$ the following holds  for any $t\ge n^{-c/ \alpha}\kappa/\sqrt{\alpha}$ 

$$\P\Big(X_{n-1} \mbox{ has two eigenvectors }v,v' \mbox{ with eigenvalues } |\lambda-\lambda'|\le t \mbox{ and there exists } u \in H_{v,v'} \cap \Incomp(c_0,c_1)$$

$$\mbox{ such that } \LCDhat_{\kappa,\gamma}(u,\alpha)\le t^{-1} \kappa /\sqrt{\alpha}\Big)  =O( \exp(-\alpha_0 n/2)).$$
\end{theorem}

\begin{proof}[Proof of Theorem \ref{theorem:Wigner:d:poor}] Let 

$$D:= t^{-1} \kappa/\sqrt{\alpha} \mbox { and } \beta:=t .$$

It is clear that with the given range of $t$ one has $1\le D\le n^{c/\alpha}$. By approximating $\lambda_i(X_{n-1})$ by $\lambda_0$, a multiple of $\beta^{-1}$, and also by \eqref{eqn:Wigner:d:CS} and by the triangle inequality

$$\|(X_{n-1}-\lambda_0)u\|\le \|(X_{n-1}-\lambda_i(X_{n-1}))u\| + |\lambda_i(X_{n-1})-\lambda_0| = O(\beta).$$ 

To complete the proof, we just apply Lemma \ref{lemma:Wigner:key} for these choices of $D$ and $\beta$, and then take the union bound over all $O(\sqrt{n}\beta^{-1})$ choices of $\lambda_0$.
\end{proof}

Let $\CE_i$ be the event that $X_{n-1}$ has eigenvectors $v,v'$ with eigenvalues $|\lambda_{i+1}(X_{n-1})-\lambda_{i}(X_{n-1})|\le t$, and $\CE_i\wedge \CH_i$ be the event $\CE_i$ coupled with $\inf_{u\in H_{v,v'}} \LCDhat_{\kappa, \gamma}(u,\alpha) \gg  t^{-1} \kappa/\alpha$. We have learned from Theorem \ref{theorem:Wigner:d:poor} that 

$$\P(\CE_i \wedge \CH_i) = \P(\CE_i) - O( \exp(-\alpha_0 n/2)).$$

On the other hand, our main result of the previous section, \eqref{eqn:Wigner:1:ball}, implies that $\P(\CE_i) =O\left(\frac{\delta}{\sqrt{\alpha}}\right)$. Thus

\begin{equation}\label{eqn:Wigner:d:EH}
\P(\CE_i \wedge \CH_i) \ll \frac{\delta}{\sqrt{\alpha}}.
\end{equation}

Now, condition on $\CE_i \wedge \CH_i$,  for every $u\in H_{v,v'}$ we have $ \LCDhat_{\kappa,\gamma}(u,\alpha)\ge t^{-1} \kappa/\sqrt{\alpha}$. It then follows from Lemma \ref{lemma:comparison:regularized} that 

$$\LCDhat_{\kappa, \gamma }(u,\alpha) \ll \sqrt{\alpha}\LCD_{\kappa,c_1\gamma \sqrt{\alpha} /2}(u).$$ 

We have thus obtained the key estimate

\begin{equation}\label{eqn:Wigner:d:vv'}
\LCD_{\kappa, c_1\gamma \sqrt{\alpha}/2}(v,v') = \inf_{u\in H_{v,v'}} \LCD_{\kappa, c_1\gamma \sqrt{\alpha}/2}(u) \gg  t^{-1} \kappa/\alpha.
\end{equation}

Now we estimate $\P(\CF_i| \CE_i\wedge \CH_i)$ (where we recall $\CF_i$ from \eqref{eqn:Wigner:d:F}). As $\delta = tn^{1/2} \gg  (t^{-1} \kappa/\alpha)^{-1}$, and as $c_1,\gamma$ are all fixed constants, by Theorem \ref{theorem:RV:d},

\begin{equation}\label{eqn:Wigner:d:FEH}
\P(\CF_i | \CE_i\wedge \CH_i) \ll \left(\frac{\delta}{\sqrt{\alpha}} + e^{-\Theta(\kappa^2)}\right)^2 \ll \left(\frac{\delta}{\sqrt{\alpha}}\right)^2.
\end{equation}

Combining  \eqref{eqn:Wigner:d:FEH} with \eqref{eqn:Wigner:d:EH}, we obtain that

\begin{equation}\label{eqn:Wigner:d:together}
\P(\CF_i \wedge \CE_i) \ll \left(\frac{\delta}{\sqrt{\alpha}}\right)^3 + \exp(-\Theta(n)) \ll \left(\frac{\delta}{\sqrt{\alpha}}\right)^3.
\end{equation}

Together with \eqref{badevent1}, this completes the proof of the case $l=2$ in Theorem \ref{theorem:main:Wigner:d}, where the condition $\delta \ge n^{1-c/\alpha}$ automatically  implies the requirement $t\ge n^{-c/ \alpha}\kappa/\sqrt{\alpha}$ of Theorem \ref{theorem:Wigner:d:poor}, as long as $c$ is sufficiently small.

\subsection{Proof sketch for general $l$} By the definition of $d$, we have  $2^d\le l < 2^{d+1}$, and that 

$$\lambda_i(X_n) \le \lambda_i(X_{n-1})\le \dots \le  \lambda_{i+l-1}(X_{n-1}) \le \lambda_{i+l}(X_n).$$ 

Let $v_1,\dots, v_{l}$ be the eigenvectors of $X_{n-1}$ corresponding to $\lambda_i(X_{n-1}),\dots,\lambda_{i+l-1}(X_{n-1})$. Then similarly to \eqref{eqn:Wigner:d:F}, with $X$ being the last column of $X_n$ in \eqref{discussion:mn-split:d}, we have the conjunction

$$(|v_1^T X |  =O(\delta) ) \wedge \dots \wedge (|v_{2^d}^T X |  = O(\delta)).$$

Also, as an analog of Fact \ref{fact:Wigner:d}, for any $u\in H_{v_1,\dots,v_{2^d}}$ and for $t=\delta/n^{1/2}$ we have

$$\|(X_{n-1}-\lambda_i(X_{n-1}))^{2^d} u\|\le t^{2^d}.$$ 

So by Cauchy-Schwarz,

\begin{equation}\label{eqn:Wigner:l:CS}
\|(X_{n-1}-\lambda_i(X_{n-1}))u\|\le t.
\end{equation}

Additionally, arguing similarly to the proof of Theorem \ref{theorem:Wigner:d:poor}, if we let $\CE_i'$ be the event that $X_{n-1}$ has eigenvectors $v_1,\dots,v_{l}$ with eigenvalues $|\lambda_{i+l-1}(X_{n-1})-\lambda_{i}(X_{n-1})|\le t$, and $\CE_i'\wedge \CH_i'$ be the event $\CE_i'$ coupled with $\inf_{u\in H_{v_1,\dots,v_{2^d}}} \LCDhat_{\kappa, \gamma}(u,\alpha) \gg  t^{-1} \kappa/\alpha$, then

$$\P(\CE_i' \wedge \CH_i') = \P(\CE_i') - O( \exp(-\alpha_0 n/2)).$$

To this end, conditioning on $\CE_i'\wedge \CH_i'$, it follows from \eqref{eqn:Wigner:l:CS}, together with the small probability bound from Theorem \ref{theorem:RV:d} (applied to $2^d$ vectors), that
$$\P_X\left (|v_1^T X |  =O(\delta) ) \wedge \dots \wedge (|v_{2^d}^T X |  = O(\delta)  )| \CE_i\wedge \CH_i \right) \ll (\frac{\delta}{\sqrt{\alpha}})^{2^d}.$$

By iterating this conditional process up to $X_{n-l}$ and proceed similarly to \eqref{eqn:Wigner:d:together}, we obtain the following lower bound for the exponent of $\delta$

$$\sum_{k=0}^{d-1} 4^k + (l-2^d+1)2^d = \frac{(3l+3 -2^{d+1})2^d -1}{3},$$

where we used the fact that the contribution to the exponent for the running index $l'$ from $2^d$ to $l$ is $(l-2^d+1)2^d$, and for $l'$ varying from $2^k$ to $2^{k+1}-1$ (with $0\le k\le d-1$) is $(2^k)(2^k) =4^k$.

\subsection{Remark}\label{remark:Wigner:1:complex} We finish this section with a short discussion on how to obtain a version of Theorem \ref{theorem:main:Wigner} for complex Wigner matrices with a probability bound of $O((\delta/\alpha)^2)$.

First of all, let $X_n$ be a random Hermitian matrix where the real and complex components of the off-diagonal terms are iid copies of a sub-gaussian random variable of mean zero and variance 1/2. 

An identical proof of Theorem \ref{theorem:Wigner:1:together} then implies that with probability $1-\exp(-\Theta(n))$, for any unit complex eigenvector $v$ of $X_n$ (including all complex phases $e^{i\theta}, 0\le \theta \le 2\pi$), the $S^{2n-1}$ vector $(\Re(v),\Im(v))$  belongs to $\Incomp(c_0,c_1)$ and

\begin{equation}\label{eqn:LCDhat:complex}
\LCDhat_{\kappa,\gamma}(\Re(v),\Im(v),\alpha)\gg n^{c/\alpha}.
\end{equation}

Next, by \eqref{badevent1}, $\P(\CE_i)$ is controlled by $\P(|v^TX| = O(\delta))$, where $v$ is any unit eigenvector associated to $\lambda_i$. For short, write $w=(\Re(v),\Im(v)), w'= (\Im(v),-\Re(v))$ and  $Y=(\Re(X),\Im(X))$, we then have

\begin{equation}\label{badevent'}
\P(|v^TX| = O(\delta)) = \P \left( |w^T Y \big| =O(\delta) \wedge  |({w'}^T Y)|=O(\delta) \right).
\end{equation}

Thus we are in the position to apply Theorem \ref{theorem:RV:d}. For this, we need to verify that the subspace $H_{w,w'}$ does not contain any unit vector with small $\LCDhat$. However, this is exactly what we have obtained in \eqref{eqn:LCDhat:complex} (see also the proof of Theorem \ref{theorem:Wigner:d:poor}). 

In conclusion, with $\delta\ge n^{-c/\alpha}$, the bound of Theorem \ref{theorem:RV:d} then implies 

$$\P(\CE_i) \ll (\frac{\delta}{\sqrt{\alpha}})^2.$$

For  the more general Theorem \ref{theorem:main:Wigner:d}, it looks plausible that we can double the value of $c_l$, using a similar, but more tedious argument.  We skip the details.

\section{Consecutive gaps for perturbed matrices: proof of Theorem \ref{theorem:main:perturbation}}\label{section:perturbation}

In the perturbed case $M_n = F_n + X_n$,  the approach used in Section \ref{section:Wigner:1} and Section \ref{section:Wigner:d} does not seem to work. This is mainly because the norm of $M_n$ now blows up, and also because $\xi$ is assumed to have bounded $(2+\ep_0)$-moment only. To handle this case, we will rely instead on the inverse Littlewood-Offord results developed by the second and third author.

It suffices to justify Theorem \ref{theorem:perturbation:poor}. As we have seen, it is technically convenient to replace the concentration probability $\rho_\delta(v)$ with a segmental variant $\rho_{\delta,\alpha}(v)$, and to work with two related scales $\delta$ rather than a single scale $\delta$. 

More precisely, for any $\delta > 0$, $0 < \alpha \leq 1$ and any $v \in \R^n$, let $\rho_{\delta,\alpha}(v)$ denote the quantity
$$ \rho_{\delta,\alpha}(v) := \inf_{I \subset \{1,\dots,n\}: |I| = \lfloor\alpha n\rfloor} \rho_\delta( v \downharpoonright_I )$$
where $v \downharpoonright_I = (v_{i_1},\dots,v_{i_m})$ is the restriction of $v$ to $I = \{i_1,\dots,i_m\}$ with $i_1 < \dots < i_m$. 

 We observe the easy inequality
\begin{equation}\label{rhoar}
\rho_\delta(v) \leq \rho_{\delta,\alpha}(v).
\end{equation}

We will assume that the matrix $M_n$ has operator norm bounded by $n^{\gamma'}$, for some $\gamma' \ge \max\{\gamma,1/2\}$, with probability at least $1-O(n^{-A})$. Thus for instance if $\xi$ has subgaussian distribution, then one can take $\gamma' =\max\{\gamma,1/2\}$. (One can also see that $\gamma' \le A/2+\gamma+4$ in any case.)

For the rest of the treatment of Theorem \ref{theorem:perturbation:poor}, we will choose 

\begin{equation}\label{eqn:eps}
\alpha:=\frac{1}{n^\eps}
\end{equation}

for some sufficiently small constant $\eps$ depending on $A$ and $\gamma'$.

We now reduce Theorem \ref{theorem:perturbation:poor} to

\begin{theorem}\label{theorem:perturbation:poor-tech}  With probability at least $1-O(\exp(-\alpha_0n))$ for some positive constant $\alpha_0$ independent of $n$, there is no unit eigenvector $v$ of $M_n$ with an eigenvalue $\lambda$ of order $O(n^{\gamma'})$ with the following property: there is a radius $\delta$ with $n^{-B}\le \delta\le n^{-B/2}$ such that

$$ n^{-A}\le \rho_{n^{\gamma'}\delta,\alpha}(v) \leq n^{0.49} \rho_{\delta,\alpha}(v).$$
\end{theorem}

To deduce Theorem \ref{theorem:perturbation:poor}, we assume otherwise that  $\rho_{n^{-B}}(v)\ge n^{-A}$. Define the sequence $(\delta_i)_{i=0}^J$ with $J := \lceil A/0.49 \rceil$, $\delta_0 := n^{-B}$, and
$$ \delta_{j+1} := n^{\gamma'}\delta_j $$
for $0 \leq j < J$.
By assumption, $\rho_{\delta_0,\alpha}(v) \ge n^{-A}$, and if we choose $B$ so that 

\begin{equation}\label{eqn:A:1}
B\ge 2J\gamma',
\end{equation}

then because the $\rho$'s are bounded by one, there exists $j< J$ (and so $n^{-B} \le \delta_{j+1} \le n^{-B/2}$) such that

$$\rho_{\delta_{j+1},\alpha}(v)\le n^{.49}\rho_{\delta_j,\alpha}(v).$$

It remains to establish Theorem \ref{theorem:perturbation:poor-tech}, this is the objective of the rest of this section.

\subsection{The compressible case}\label{subsection:perturbation:compressible}
 Let us first consider the easy case in which there exists $n^{-B}\le \delta \le n^{-B/2}$ such that $\rho_{\delta,\alpha}(v)\ge (\alpha n)^{-1/2+\eps}$, this case is similar to our treatment of Subsection \ref{subsection:Wigner:compressible}. Our main tool will be the following Erd\H{o}s-type inverse Littlewood-Offord theorem:

\begin{theorem}\cite{erdos}\label{theorem:cilf-erdos}  Let $\eps>0$ be fixed, let $\delta > 0$, and let $v \in \R^m$ be a unit vector with
$$ \rho_\delta(v) \geq m^{-\frac{1}{2}+\eps}.$$
Then all but at most $\eps m$ of the coefficients of $v$ have magnitude at most $\delta$.
\end{theorem}

Now as $ q=\rho_{\delta,\alpha}(v) > (\alpha n)^{-1/2+\eps}$, by Theorem \ref{theorem:cilf-erdos}, we see that for every $I \subset \{1,\dots,n\}$ with $|I| = \lfloor \alpha n \rfloor$, all but at most $O( \eps |I| )$ of the coefficients of $v \downharpoonright_I$ have magnitude at most $\delta$, where the implied constant does not depend on $\eps$.  By a simple covering argument, we conclude that $|v_i| \leq \delta$ for all $i$ outside of an exceptional set $S \subset \{1,\dots,n\}$ of cardinality $|S| = O( \eps \alpha n ) =O(n^{1-\eps})$.

If we let $\CE_S$ denote the event that the above situation holds for a given value of $S$, the probability that the conclusion of Theorem \ref{theorem:perturbation:poor-tech} fails may thus be bounded via the union bound by
$$ \sum_{S \subset \{1,\dots,n\}: |S| = O(\alpha n)} \P( \CE_S ) + O( \exp(-\Theta(n)) ).$$
By Stirling's formula, the number of possible exceptional sets $S$ is at most $n^{O(\eps n^{1-\eps})}$.  Thus it suffices to show that

\begin{equation}\label{perturbation:pes}
 \P( \CE_S ) \ll \exp( -\alpha_0 n  )
\end{equation}

uniformly in $S$ and some fixed $\alpha_0>0$ independent of $\eps$.

By symmetry we may take
$$ S = \{n'+1,\dots,n\}$$
for some $n' = (1 - O(\alpha)) n$.

Now suppose that the event $\CE_S$ holds, and let $v$, $\lambda$ be as above.  We split
$$ M_n = \begin{pmatrix} M_{n'} & D \\ D^* & M_{n-n'} \end{pmatrix}$$
where $M_{n'}$, $M_{n-n'}$ are the top left $n' \times n'$ and bottom right $(n-n' )\times (n-n')$ minors of $M_n$ respectively, and $D$ is a $n' \times (n-n')$ matrix whose entries are independent copies of $\xi$ (modulo the deterministic part from $F_n$), and also independent of $M_{n'}, M_{n-n'}$.  We also split $ v^T = (v' , v'')^T$, with $v' \in \R^{n'}$ and $v'' \in \R^{n-n'}$.  

By hypothesis, all entries of $v'$ have magnitude at most $\delta=O(n^{-B/2})$, and so provided that $B>1$,

\begin{equation}\label{perturbation:spin}
 \| v' \| \leq \delta n^{1/2}\ll \frac{1}{10}
\end{equation}

and thus (since $v$ is a unit vector)

\begin{equation}\label{perturbation:vpp}
\frac{1}{2} \leq \|v''\| \leq 1.
\end{equation}

From the eigenvector equation
$$
\begin{pmatrix} M_{n'} & D \\ D^* & M_{n-n'} \end{pmatrix} 
\begin{pmatrix} v' \\ v'' \end{pmatrix} = \lambda \begin{pmatrix} v' \\ v'' \end{pmatrix}$$

we have

\begin{equation}\label{perturbation:kitty}
 (M_{n'}-\lambda) v' + D v'' = 0.
\end{equation}

Hence,

\begin{equation}\label{perturbation:bojo}
 \|  Dv'' \| =O( n^{\gamma'} \delta n^{\frac{1}{2}}) =O( n^{-\frac{B}{2} + \frac{1}{2} +\gamma'}).
\end{equation}

On the other hand, by a standard epsilon-net argument (noting that $n-n'=O(n^{1-\eps})$), with probability $1-O(\exp(-\alpha_0 n))$ for some fixed $\alpha_0>0$ independent of $\eps$, we have

\begin{equation}\label{perturbation:dvin}
 \inf_{w \in\R^{n-n'}: \|w\|=1} \| Dw \| \gg n^{-\frac{1}{2}}.
\end{equation}

We hence obtain \eqref{perturbation:pes} provided that  

\begin{equation}\label{eqn:A:2}
B> 2\gamma' +2.
\end{equation}

\subsection{The incompressible case}\label{subsection:perturbation:incompressible}

Now we assume that there exists a radius $\delta$ with $n^{-B}\le \delta\le n^{-B/2}$ such that

$$ n^{-A}\le  \rho_{n^{\gamma'}\delta,\alpha}(v) \leq n^{.49} \rho_{\delta,\alpha}(v); \mbox{ and } \rho_{\delta,\alpha} \le (\alpha n)^{-1/2+\eps}.$$

In this case, Theorem \ref{theorem:cilf-erdos} is insufficient to control the ''rich" vector $v$.  Instead, we will rely on the more general inverse Littlewood-Offord theorem from \cite{NVoptimal} (see also \cite{TVinverse,TVinverse1}).  Define a \emph{generalised arithmetic progression} (or GAP) to be a finite subset $Q$ of the real line $\R$ of the form
$$ Q = \{ a_1 w_1 + \dots + a_r w_r: a_i \in \Z, |a_i| \leq N_i \hbox{ for all } i=1,\dots,r\}$$
where $r \geq 0$ is a natural number (the \emph{rank} of the GAP), $N_1,\dots,N_r > 0$ are positive integers (the \emph{dimensions} of the GAP), and $w_1,\dots,w_r \in \R$ are real numbers (the \emph{generators} of the GAP).  We refer to the quantity $\prod_{i=1}^r (2N_i+1)$ as the \emph{volume} $\operatorname{vol}(Q)$ of $Q$; this is an upper bound for the cardinality $|Q|$ of $Q$. We then have

\begin{theorem}[Continuous inverse Littlewood-Offord theorem]\label{theorem:cilf}  Let $\eps>0$ be fixed, let $\delta > 0$, and let $v \in \R^n$ be a unit vector whose small ball probability $\rho := \rho_{\delta}(v)$ obeys the lower bound
$$ \rho \gg n^{-O(1)}.$$
Then there exists a generalized arithmetic progression $Q$ of volume 
$$ \operatorname{vol}(Q) \leq \max\left( O\left(\frac{1}{\rho\sqrt{n}}\right), 1 \right)$$
such that all but at most $\eps n$ of the coefficients $v_1,\dots,v_n$ of $v$ lie within $\delta$ of $Q$.  Furthermore, if $r$ denotes the rank of $Q$, then $r=O(1)$, and all the generators $w_1,\dots,w_r$ of $Q$ have magnitude $O(1)$.
\end{theorem}

We now begin the proof of Theorem \ref{theorem:perturbation:poor-tech} with the setting

$$q := \rho_{\delta,\alpha}(v)  < (\alpha n)^{-1/2+\eps}.$$ 

 As $ \rho_{n^{\gamma'} \delta,\alpha}(v) \leq n^{.49} q$, there must exist a subset $I$ of $\{1,\dots,n\}$ of cardinality $|I| = \lfloor \alpha n \rfloor$
with

\begin{equation}\label{perturbation:rory}
 \rho_{n^{\gamma'} \delta}(v \downharpoonright_I ) \leq n^{.49} q.
\end{equation}

For each $I$, let $\CE'_I$ be the event that the above situation occurs, thus the conclusion of Theorem \ref{theorem:perturbation:poor-tech} holds with probability at most
$$ \sum_{I \subset \{1,\dots,n\}: |I| = \lfloor \alpha n \rfloor} \P( \CE'_I ) + O( \exp(-\Theta(n) ) ).$$
We crudely bound the number of possible $I$ by $2^n$. The key estimate is stated below.

\begin{lemma}\label{lemma:perturbation:E'_I} We have
$$\P(\CE'_I) \ll n^{-0.1 n + O( \eps n )}.$$
\end{lemma}

This will establish Theorem \ref{theorem:perturbation:poor-tech} in the incompressible case if $\eps$ is chosen small enough.

Our approach to prove Lemma \ref{lemma:perturbation:E'_I} is somewhat similar to \cite{TVcomp}. First, by symmetry we may assume that $I = \{1,\dots,k\}$, where $k := \lfloor \alpha n \rfloor$. Similarly to the previous section, we split

$$ M_n = \begin{pmatrix} M_k & D \\ D^* & M_{n-k} \end{pmatrix}$$

where $M_{k}$, $M_{n-k}$ are the top left $k \times k$ and bottom right $(n-k) \times (n-k)$ minors of $M_n$ respectively, and $D$ is a $k \times (n-k)$ matrix whose entries are (modulo the deterministic part from $F_n$) independent copies of $\xi$, and also independent of $M_k, M_{n-k}$.  We also split $ v^T = (v' , v'')^T $ with $v' \in \R^{k}$ and $v'' \in \R^{n-k}$.

Now suppose that $M_n$ is such that the event $\CE'_I$ holds, the heart of our analysis is to approximate $(v,\lambda)$ by low-entropy structures.

\begin{lemma}\label{lemma:perturbation:approx}  There exists a subset $\CN$ of $\R^n \times \R^k \times \R$ of size $O(n^{-n/2 + O(\alpha n)} q^{-n} )$ such that for any eigenvector-eigenvalue pair $(v,\lambda)$ satisfying $\CE_I'$, there exists $(\tilde v, w', \tilde \lambda)\in \CN$ which well-approximates $(v,\lambda)$ in the following sense
\begin{itemize}
\item (general approximation of eigenvector) $ |v_j - \tilde v_j| \le \delta$, for $1\le j\le n$;
\vskip .1in
\item  (finer approximation over segment) $|v'_j - w'_j| \le n^{-\gamma'-1} \delta$, for $1\le j\le k$;
\vskip .1in
\item (approximation of eigenvalue) $|\lambda - \tilde \lambda| \leq n^{-\gamma'-1} \delta$.
\end{itemize}
\end{lemma}

We also refer the reader to Lemma \ref{lemma:Wigner:structure:1} in Appendix \ref{section:Wigner:key} for similarities.

\begin{proof}[Proof of Lemma \ref{lemma:perturbation:approx}]
 We cover $\{1,\dots,n\}$ by sets $I_1,\dots,I_m$ of of length differing by at most one, with 
 
 $$m \leq \left\lfloor \frac{1}{\alpha} \right\rfloor + 1 = O(n^{\eps}).$$  
 
 Because $\rho_\delta(v)=q$, for each $i=1,\dots,m$ we have

$$ \rho_\delta( v \downharpoonright_{I_i} ) \ge q .$$

Applying Theorem \ref{theorem:cilf} and the incompressibility hypothesis $q < (\alpha n)^{-1/2+\eps}$, we may thus find, for each $i=1\dots,m$, a GAP $Q_i$ such that

\begin{equation}\label{perturbation:volq}
 \operatorname{vol}(Q_i) \ll (\alpha n)^{-1/2+\eps} / q
\end{equation}

and rank $r_i = O(1)$ such that all but at most $O(\eps^2 n)$ of the coefficients of $v \downharpoonright_{I_i}$ lie within a distance $O(\delta)$ of $Q_i$. Thus we have

$$ \operatorname{dist}( v_j, \bigcup_{i=1}^m Q_i ) \ll \delta$$

for all $j=1,\dots,n$ outside of an exceptional set $S \subset \{1,\dots,n\}$ of cardinality $|S| \ll \eps n$.  

Furthermore, all the generators $w_{i,1},\dots,w_{i,r_i}$ of $Q_i$ have magnitude $O(1)$.  From \eqref{perturbation:volq} we have the crude bound for the dimensions of $Q_i$

$$ N_{i,1},\dots,N_{i,r_i} \ll (\alpha n)^{-1/2+\eps}/q.$$

From this, we may round each generator $w_{i,l}$ to the integer nearest multiple of $q\delta$ (say) without loss of generality, since this only moves the elements of $Q_i$ by $O( (\alpha n)^{-1/2+\eps} \delta ) = O(\delta)$ at most.  In particular, all elements of $\bigcup_{i=1}^m Q_i$ are now integer multiples of $q\delta$.

Thanks to this information, one can create a coarse ``discretized approximation'' $\tilde v = (\tilde v_1,\dots,\tilde v_n)$ to $v$, by setting $\tilde v_j$ for $j=1,\dots,n$ to be the nearest element of $\bigcup_{i=1}^m Q_i$ to $v_j$ if $v_j$ lies within $O(\delta)$ of this set, or the nearest multiple of $q\delta$ to $v_j$ otherwise.  Then $\tilde v_j$ consists entirely of multiples of $q\delta$, one has

\begin{equation}\label{perturbation:vob}
 |v_j - \tilde v_j| \le \delta
\end{equation}

for all $j=1,\dots,n$, and all but at most $O(\alpha n)$ of the coefficients of $\tilde v$ lie in $\bigcup_{i=1}^m Q_i$.  

We will also need a finer approximation $w' = (w_1,\dots,w_k)^T$ to the component $v' =(v_1,\dots,v_k)^T$ of $v$, by choosing $w_j$ to be the nearest integer multiple of $n^{-\gamma'} \delta$ to $v_j$, thus for all $j=1,\dots,k$. 

\begin{equation}\label{perturbation:wob}
 |v'_j - w'_j| \le n^{-\gamma'} \delta.
\end{equation}

Similarly, one approximates the eigenvector $\lambda$ by the nearest multiple $\tilde \lambda$ of $n^{-\gamma'-1} \delta$, thus

\begin{equation}\label{perturbation:lob}
|\lambda - \tilde \lambda| \leq n^{-\gamma'-1} \delta.
\end{equation}



We now claim that the data $\tilde v$, $w'$, $\tilde \lambda$ have low entropy, in the sense that they take a relatively small number of values.  Indeed, the number of possible ranks $r_1,\dots,r_m$ of the $Q_i$ is $(O(1))^m =(O(1))^{1/\alpha}$.  Once the ranks are chosen, the number of ways we can choose the generators $w_{i,l}$ of the $Q_i$ (which are all multiples of $q\delta$ of magnitude $O(1)$) are

$$ O\big(( 1/q\delta )^{\sum_{i=1}^m r_i}\big) = O( n^{O(1/\alpha)} ) = O(n^{O(n^\eps)})$$

since $q, \delta \geq n^{-O(1)}$ by hypothesis.  

The number of sets $S$ may be crudely bounded by $2^n$.  For each $j \in S$, the number of choices for $\tilde v_j$ is $O(1/q\delta)$, leading to $O((1/q\delta)^{|S|}) = O((1/q\delta)^{O(\alpha n)})$ ways to choose this portion of $\tilde v$.  For $j \not \in S$, $\tilde v_j$ lies in $\bigcup_{i=1}^m Q_i$, which has cardinality $O( m \operatorname{vol}(Q) ) = O( \frac{1}{\alpha} n^{-1/2+\alpha} / q )$.  Finally, for the $k = \lfloor \alpha n \rfloor$ coefficients of $w'$ there are at most $O( n^{\gamma'+1}/\delta )$ choices, and there are similarly $O(n^{\gamma'+1}/\delta)$ choices for $\tilde \lambda$.

Thus the total possible number of quadruples\footnote{One could also expand this set of data by also adding in the GAPs $Q_1,\dots,Q_m$ and the set $S$, but we will have no need of this additional data in the rest of the argument.} $(\tilde v,w',\tilde \lambda)$ is at most

\begin{equation}\label{note}
 (O(1))^{1/\alpha} \times O(n^{O(1/\alpha)} ) \times 2^n \times O\big((1/q\delta)^{O(\alpha n)}\big) \times O\big(( \frac{1}{\alpha} n^{-1/2+\alpha}/q )^n\big) \times O\big(( n^{\gamma'+1}/q )^{\lfloor \alpha n \rfloor}\big) \times O(n^{100}/q )^2,
\end{equation}

which simplifies to

$$ O( n^{-n/2 + O(\alpha n)} q^{-n} ).$$

Note that many of the factors in the previous expression \eqref{note} can be absorbed into the $O(n^{O(\alpha n)})$ error as $\eps$ is chosen sufficiently small.

\end{proof}

Now we assume that $(v,\lambda)$ is an eigenvector-eigenvalue pair satisfying $\CE_I'$, which can be approximated by tuple $(\tilde v, w', \tilde \lambda)\in \CN$ as in Lemma \ref{lemma:perturbation:approx}. The lower $n-k$-dimensional component of the eigenvector equation $M_n v = \lambda v$ then reads as

$$ D^* v' + (M_{n-k}-\lambda) v'' = 0.$$

From \eqref{perturbation:wob} and  \eqref{perturbation:lob} we certainly have

$$ \| (\tilde \lambda - \lambda) v'' \| \ll \delta \mbox{ and } \| D^* (v' - w' ) \| \ll \delta .$$
  
By \eqref{perturbation:vob}, we have

$$ \| v'' - \tilde v'' \| \leq n^{\frac{1}{2}} \delta$$

where $\tilde v''$ is the lower $n-k$ entries of $\tilde v$.  

Since $\lambda$ has order $O(n^{\gamma'})$, we have $\tilde \lambda = O(n^{\gamma'})$, and so 

$$\|\tilde \lambda(v'' - \tilde v'')\| \le n^{\gamma' +\frac{1}{2}}\delta.$$

Hence, with $u= (M_{n-k}-\tilde \lambda) \tilde v''$ independent of $D$, 

\begin{align}\label{perturbation:ineq}
 \|D^* w' -u\|& = \| D^* w' - (M_{n-k}-\tilde \lambda) \tilde v''  \| \nonumber \\ 
 &= \| D^* (w'-v') + D^*v' + (M_{n-k}-\lambda) v'' +   (M_{n-k}-\lambda) (\tilde v''-v'') + (\lambda -\tilde \lambda) \tilde v''\| \nonumber \\
 & \le n^{\gamma'+\frac{1}{2}} \delta.
\end{align}

Let us now condition $M_{n-k}$ to be fixed, so that $u$ is deterministic.  Let $x_1,\dots, x_{n-k} \in \R^k$ denote the rows of $D$; then the $x_1,\dots,x_{n-k}$ are independent vectors, each of whose elements (up to a deterministic part from $F_n$) is an independent copy of $\xi$.  The bound \eqref{perturbation:ineq} then can be rewritten as

\begin{equation}\label{perturbation:ineq-2}
 \sum_{i=1}^{n-k} |x_i^T w' - u_i|^2 \le n^{2\gamma'+1} \delta^2
\end{equation}
where $u_1,\dots,u_{n-k} \in \R$ are the coefficients of $u$.

In summary, if we let $\CE_{\tilde v, w', \tilde \lambda}$ be the event that the above situation holds for a \emph{given} choice of $\tilde v, w', \tilde \lambda$, it will suffice by the union bound (taking into account the cardinality of $\CN$ from Lemma \ref{lemma:perturbation:approx}) to establish the following upper bound

\begin{lemma}\label{lemma:tildeE} For any $(\tilde v, w', \tilde \lambda)\in \CN$,
\begin{equation}\label{perturbation:flong}
\P(\CE_{\tilde v, w', \tilde \lambda}) =O(( n^{.49 +O(\alpha)}q)^n).
\end{equation}
\end{lemma}

To justify \eqref{perturbation:flong}, we first recall from \eqref{perturbation:rory} that $\rho_{n^{\gamma'} \delta}(v' ) \leq n^{.49} q$. From \eqref{perturbation:wob}, the random walks (with the $x_i$) associated to $v'$ and $w'$ differ by at most $O(n^{-\gamma'+\frac{1}{2}} \delta)$, and so
$$
 \rho_{n^{\gamma'} \delta/2}(w' ) \leq n^{.49}q.
$$

We now invoke the following tensorization trick (which is not strictly necessary here, but will be useful later).

\begin{lemma}\cite[Lemma 2.2]{rv}\label{lemma:tensorization}  Let $\zeta_1,\dots,\zeta_n$ be independent non-negative radom variables, and let $K, t_0 > 0$.  If one has
$$ \P( \zeta_k < t ) \leq K t$$
for all $k=1,\dots,n$ and all $t \geq t_0$, then one has
$$ \P( \sum_{k=1}^n \zeta_k^2 < t^2 n ) \leq O((K t)^n)$$
for all $t \geq t_0$.
\end{lemma}

By setting $t= n^{\gamma'} \delta/2$, it follows crudely from Lemma \ref{lemma:tensorization} that
$$ \P\Big( \sum_{i=1}^{n-k} |x_i^T  w' - v_i|^2 =O(n^{2\gamma'+1} \delta^2) \Big) < \P\Big( \sum_{i=1}^{n-k} |x_i^T w' - v_i|^2 =o( nt^2) \Big)  = O\big(( n^{.49} q )^{n-k}\big),$$
where as before we are conditioning on $M_{n-k}$ being fixed.  

Undoing the conditioning, we conclude that \eqref{perturbation:ineq-2}, and hence \eqref{perturbation:ineq}, occurs with
probability $O(( n^{.49} q )^{n-k}) = O(( n^{.49+O(\alpha)} q )^{n})$, and thus \eqref{perturbation:flong} follows.  This concludes the proof of Theorem \ref{theorem:perturbation:poor-tech} in the incompressible case.

Finally, the conditions for $B$ from \eqref{eqn:A:1} and \eqref{eqn:A:2} can be secured by choosing

\begin{equation}\label{eqn:A}
B> 5A\gamma'+ 2\gamma'+2.
\end{equation}

\section{Consecutive gaps for Erd\H{o}s-R\'enyi graphs: proof of Theorem \ref{theorem:main:ER}}\label{section:ER}

By modifying the treatment of Section \ref{section:perturbation} and using \eqref{badevent1}, we will show the following.

\begin{theorem}[Most eigenvectors poor]\label{theorem:ER:poor}   Let $A > 0$ and $\sigma>0$ be fixed.  Then, with probability at least $1 - O(\exp(-\alpha_0 n))$ for some positive constant $\alpha_0>0$ (which may depend on $A,\sigma$), every unit eigenvector $v$ of $A_n$ with eigenvalue $\lambda$ in the interval $[-10\sqrt{n}, 10\sqrt{n}]$ obeys the concentration estimate
$$ \rho_{\delta}(v) \leq n^\sigma \delta$$
for all $\delta > n^{-A}$.
\end{theorem}

In fact, we shall reduce Theorem \ref{theorem:ER:poor} to

\begin{theorem}\label{theorem:ER:poor:tech}  Let $\sigma > 0$ be fixed, and let $\delta \geq n^{-A}$ and $q \geq n^\sigma \delta$.  Let $\alpha>0$ be a sufficiently small fixed quantity.  Then with probability at least $1-O(\exp(-\alpha_0n))$ for some positive constant $\alpha_0>0$, every unit eigenvector $v$ of $A_n$ with an eigenvalue $\lambda$ in the interval $[-10\sqrt{n}, 10\sqrt{n}]$, one either has
$$ \rho_{\delta,\alpha}(v) \leq q$$
or
\begin{equation}\label{ER:dance}
 \rho_{n^{1/2+\alpha} \delta,\alpha}(v) > n^{1/2-\sqrt{\alpha}} q.
\end{equation}
\end{theorem}

We remark that, unlike in Section \ref{section:perturbation}, $\alpha$ is a (sufficiently small) constant here. It is not hard to see that Theorem \ref{theorem:ER:poor:tech} implies Theorem \ref{theorem:ER:poor}. Indeed, let $A_n,A, \sigma$ be as in Theorem \ref{theorem:ER:poor}.  We may assume $\delta \leq 1$, as the claim is trivial for $\delta>1$.  By rounding $\delta$ up to the nearest power of $n^{-\sigma/2}$, and using the union bound, it suffices to show that for each given such choice of $n^{-A} < \delta \leq 1$,  with probability  $1-O(\exp(-\alpha_0 n))$ all unit eigenvectors $v$ of $A_n$ with eigenvalue in the interval $[-10\sqrt{n}, 10\sqrt{n}]$ satisfy the bound

\begin{equation}\label{ER:rhovs}
 \rho_{\delta}(v) \leq n^{\sigma/2} \delta.
\end{equation}

Let $\alpha>0$ be a sufficiently small fixed quantity, and set $J := \lfloor 2A \rfloor$. For each $j=0,\dots,J+1$, we define the quantities

$$\delta_j := n^{(1/2+\alpha)j} \delta \label{ER:rd}; \mbox{ and }
K_j := n^{\sigma/2 - (\sqrt{\alpha}-\alpha)j} \label{ER:kd}.$$

Now suppose that \eqref{ER:rhovs} fails.  We conclude that

$$ \rho_{\delta_0,\alpha}(v) \ge \rho_{\delta_0}(v)> K_0 \delta_0.$$

On the other hand, if $\alpha$ is sufficiently small depending on $\sigma, A$, 

$$ \rho_{\delta_J,\alpha}(v) \leq 1 \leq \delta_J \leq K_J \delta_J.$$

Thus there exists $0 \leq j < J$ such that

$$
 \rho_{\delta_j,\alpha}(v) > K_j \delta_j; \mbox{ and }  \rho_{\delta_{j+1},\alpha}(v) \leq K_{j+1} \delta_{j+1}.
$$

Consequently,

$$
\rho_{n^{1/2+\alpha} \delta_j,\alpha}(v) \leq n^{1/2-\sqrt{\alpha}} K_j \delta_j.
$$

But by Theorem \ref{theorem:ER:poor:tech} (with $\sigma$ replaced by $\sigma/4$, say, and $\delta$ and $q$ replaced by $\delta_j$ and $K_j \delta_j$) and by the union bound over the $J$ different choices of $j$, this event can only occur with probability $O( \exp(-\alpha_0 n) )$ for some positive constant $\alpha_0$, and Theorem \ref{theorem:ER:poor} follows.

In what follows we will establish Theorem \ref{theorem:ER:poor:tech}.  Again, similarly to the proof of Theorem \ref{theorem:perturbation:poor-tech}, we divide into two cases.

\subsection{The compressible case}\label{subsection:ER:compressible}
 Assume that $q \geq n^{-1/2+\alpha}$ and $\rho_{\delta,\alpha}(v)>q$. By a simple covering argument, we have
 
 $$\rho_{\delta n^{-1/2+\alpha}/q, \alpha}(v)\ge n^{-1/2+\alpha}.$$
 
Set $\delta':=\delta n^{-1/2+\alpha}/q$. Proceed similarly to Subsection \ref{subsection:perturbation:compressible}, we just need to show

\begin{equation}\label{ER:pes}
 \P( \CE_S ) =O( \exp( -\alpha_0 n  )),
\end{equation}

uniformly in $S\subset [n]$ with $|S|=\lfloor \alpha n \rfloor$  and some fixed $\alpha_0>0$ independent of $\alpha$. Here $\CE_S$ is the event $\rho_{\delta'}(v\downharpoonright S)\ge n^{-1/2+\alpha}$.

By symmetry we may take $ S = \{n'+1,\dots,n\}$ for some $n' = (1 - O(\alpha)) n$. Now suppose that the event $\CE_S$ holds, and let $v$, $\lambda$ be as above. We split 

$$ A_n = \begin{pmatrix} A_{n'} & D \\ D^* & A_{n-n'} \end{pmatrix} \mbox{ and }  v = \begin{pmatrix} v' \\ v'' \end{pmatrix},$$ 

where $A_{n'}$, $A_{n-n'}$ are the top left $n' \times n'$ and bottom right $(n-n') \times (n-n')$ minors of $A_n$ respectively, and  $v' \in \R^{n'}$ and $v'' \in \R^{n-n'}$, and the entries of $v'$ are bounded by $\delta'$ by applying Theorem \ref{theorem:cilf-erdos}. Again, as in Subsection \ref{subsection:perturbation:compressible}, we have  $\frac{1}{2} \leq \|v''\| \leq 1$, and the eigenvector equation $A_nv=\lambda v$ implies that



\begin{equation}\label{ER:bojo}
 \| \tilde Dv'' - a {\mathbf 1}_{n'} \| =O(n\delta')=O( n^{1/2+\alpha} \delta / q) = O(n^{1/2+\alpha -\sigma}) =o(n^{1/2})
\end{equation}

for some scalar $a \in \R$, where $\tilde D = D - p {\mathbf 1}_{n'} {\mathbf 1}_{n-n'}^T$ is the mean zero normalization of $D$, and we note that by \eqref{eqn:Xn:norm} and by the bound on $\lambda$, 

$$ \| A_{n'} - \lambda - p {\mathbf 1}_{n'} {\mathbf 1}_{n'}^T \| =O( \sqrt{n}).$$

However, with probability $1-O(\exp(-\alpha_0 n))$ for some positive constant $\alpha_0$ independent of $\alpha$, we easily have

\begin{equation}\label{ER:dvin}
 \inf_{w \in\R^{n-n'}: \|w\|=1} \inf_{a \in \R} \| \tilde Dw - a {\mathbf 1}_{n'} \| \gg \sqrt{n}.
\end{equation}

In fact, without the $a {\mathbf 1}$ term, as noted at the end of Subsection \ref{subsection:perturbation:compressible}, the left-hand side becomes the least singular value of $\tilde D$, and the result follows from \cite[Theorem 3.1]{litvak}.  The epsilon net arguments used there can be easily modified to handle the additional $a {\mathbf 1_{n'}}$ term.  Alternatively, we can argue as follows.  We group the $n'$ rows of $\tilde D$ into $\lfloor \frac{n'}{2} \rfloor$ pairs, possibly plus a remainder row which we simply discard.  For each such pair of rows, we subtract the first row in the pair from the second, and let $D'$ be the $\lfloor \frac{n'}{2} \rfloor \times n-n'$ matrix formed by these differences.  From the triangle inequality we see that $ \| D' w \| \leq 2 \| \tilde Dw - a {\mathbf 1} \|$ for all $w \in \R^{n-n'}$ and $a \in \R$.  From \cite[Theorem 3.1]{litvak} we have
$$ \inf_{w \in \R^{n-n'}: \|w\|=1} \|D'w\| \gg \sqrt{n}$$
with probability $1-O(\exp(-\alpha_0n))$, and the claim \eqref{ER:dvin} follows.

Comparing \eqref{ER:dvin} with \eqref{ER:bojo} (after rescaleing $v''$ to be of unit length) and with \eqref{ER:rhovs}, we obtain a contradiction with probability $1-O(\exp(-\alpha_0n))$, and hence \eqref{ER:pes} follows.

\subsection{The incompressible case}\label{subsection:ER:incompressible}
 In this case there is a unit eigenvector $v = (v_1,\dots,v_n)^T$ of $A_n$ with eigenvalue $\lambda \in [-10\sqrt{n},10\sqrt{n}]$ so that

\begin{equation}\label{ER:rjv}
 \rho_{\delta,\alpha}(v) > q; \mbox{ and } \rho_{n^{1/2+\alpha} \delta,\alpha}(v) \leq n^{1/2-\sqrt{\alpha}} q.
\end{equation}
 
 From the second inequality, there must exist a subset $I$ of $\{1,\dots,n\}$ of cardinality $|I| = \lfloor\alpha n \rfloor$
with

\begin{equation}\label{ER:rory}
 \rho_{n^{1/2+\alpha} \delta}(v \downharpoonright_I ) \leq n^{1/2-\sqrt{\alpha}} q.
\end{equation}

For each $I$, let $\CE'_I$ be the event that the above situation occurs, thus the conclusion of Theorem \ref{theorem:ER:poor:tech} holds with probability at most

$$ \sum_{I \subset \{1,\dots,n\}: |I| = \lfloor \alpha n \rfloor} \P( \CE'_I ) + O( \exp(-\Theta(n) ) ).$$

It suffices to justify the following improvement of Lemma \ref{lemma:perturbation:E'_I} for $\CE'_I$.

\begin{lemma}\label{lemma:ER:E'_I} We have
$$\P(\CE'_I) \ll n^{-\sqrt{\alpha} n + O( \alpha n )}$$
where the implied constant in the $O(.)$ notation is independent of $\alpha$.
\end{lemma}

By symmetry we may assume that $ I = \{1,\dots,k\}$, where $k := \lfloor \alpha n \rfloor$. Again, we split
$$ A_n = \begin{pmatrix} A_k & D \\ D^* & A_{n-k} \end{pmatrix}$$
where $A_{k}$, $A_{n-k}$ are the top left $k \times k$ and bottom right $(n-k) \times (n-k)$ minors of $A_n$ respectively.  We also split $ v^T = (v' , v'')^T$ with $v' \in \R^{k}$ and $v'' \in \R^{n-k}$.

Now suppose that $A_n$ is such that $\CE'_I$ holds. Then by invoking Theorem \ref{theorem:cilf}, we obtain the following variant of Lemma \ref{lemma:perturbation:approx} (where $\gamma$ is replaced by, say $100$).

\begin{lemma}\label{lemma:ER:approx}  There exists a subset $\CN$ of $\R^n \times \R^k \times \R$ of size $O(n^{-n/2 + O(\alpha n)} q^{-n} )$ such that for any eigenvector-eigenvalue pair $(v,\lambda)$ satisfying $E_I'$, there exists $(\tilde v, w', \tilde \lambda)\in \CN$ which approximates $(v,\lambda)$ as follows
\begin{itemize}
\item $ |v_j - \tilde v_j| \le \delta$ for $1\le j\le n$;
\vskip .1in
\item  $|v'_j - w'_j| \le n^{-100} \delta$ for $1\le j\le k$;
\vskip .1in
\item $|\lambda - \tilde \lambda| \leq n^{-100} \delta$.
\end{itemize}
\end{lemma}

\begin{proof}[Proof of Lemma \ref{lemma:ER:E'_I}] Let us now assume that $(v,\lambda)$ can be approximated by tuple $(\tilde v, w', \tilde \lambda)\in \CN$ as in Lemma \ref{lemma:ER:approx}. The partial eigenvector equation $D^* v' + (A_{n-k}-\lambda) v'' = 0$ would then imply that (assuming \eqref{eqn:Xn:norm})

\begin{equation}\label{ER:ineq}
 \| D^* w' - u\| = O (n \delta)
\end{equation}

where $u \in \R^{n-k}$ is the vector $(A_{n-k}-\tilde \lambda) \tilde v''$.



We condition $A_{n-k}$ to be fixed, so that $u$ is deterministic.  Let $x_1,\dots,x_{n-k} \in \R^k$ denote the rows of $D$.  The bound \eqref{ER:ineq} then can be rewritten as

\begin{equation}\label{ER:ineq-2}
 \sum_{i=1}^{n-k} |x_i^T w' - u_i|^2 =O( n^2 \delta^2)
\end{equation}

where $u_1,\dots, u_{n-k} \in \R$ are the coefficients of $u$.

Now, by \eqref{ER:rory}, $ \rho_{n^{1/2+\alpha} \delta}(v' ) \leq n^{1/2-\sqrt{\alpha}} q$; combine this with the approximation from Lemma \ref{lemma:ER:approx}, we obtain

$$
 \rho_{n^{1/2+\alpha} \delta/2}(w' ) \leq n^{1/2-\sqrt{\alpha}} q.
$$

Thus, by a simple covering argument, for every $\delta'' \geq n^{1/2+\alpha} \delta/2$ one has
 
 \begin{equation}\label{ER:xiw}
 \P( |x_i^T w' - v_i| \leq \delta'' ) =O\Big( n^{-\sqrt{\alpha} - \alpha} \delta'' q / \delta\Big), \forall i=1,\dots,n-k.
\end{equation}


It follows from \eqref{ER:xiw} and from Lemma \ref{lemma:tensorization} that  for every $\delta'' \geq n^{1/2+\alpha} \delta/2$

$$ \P\Big( \sum_{i=1}^{n-k} |x_i^T w' - v_i|^2 \leq {\delta''}^2 (n-k) \Big) = O\Big((n^{-\sqrt{\alpha} - \alpha} \delta'' q / \delta)^{n-k}\Big)$$

In particular, with $\delta'' := n^{1/2+\alpha} \delta/2$, and by discarding a factor of $n^\alpha$,

$$ \P\Big( \sum_{i=1}^{n-k} |x_i^T w' - v_i|^2 =O( n^2 \delta^2) \Big) = O\Big(( n^{1/2-\sqrt{\alpha}} q )^{n-k}\Big),$$

where we are conditioning on $A_{n-k}$ being fixed.  Undoing the conditioning, we conclude that \eqref{ER:ineq-2} and hence \eqref{ER:ineq} occurs with
probability $O(( n^{1/2-\sqrt{\alpha}} q )^{n-k}) = O(( n^{1/2-\sqrt{\alpha} + O(\alpha)} q )^{n})$.  This concludes the proof of Theorem \ref{theorem:ER:poor:tech} in the incompressible case by the union bound over all tuples $(\tilde v, w',\tilde \lambda)$ from $\CN$.

\end{proof}

\section{Application: non-degeneration of eigenvectors}\label{section:nodal}

Using the repulsion between eigenvalues, we can show that the entries of eigenvectors are rarely very close to zero.

First of all, Theorem \ref{theorem:main:perturbation} easily implies the following weak stability result.

\begin{lemma}\label{lemma:stable} For any $A>0$, there exists $B>0$ depending on $A$ and $\gamma$ such that the following holds with probability at least $1-O(n^{-A})$: if there exist $\lambda \in \R$ and $v\in S^{n-1}$  such that $\|(M_n-\lambda)v\|\le n^{-B}$, then $M_n$ has an eigenvalue $\lambda_{i_0}$ with a unit eigenvector $u_{i_0}$ such that 
$$|\lambda_{i_0}-\lambda| \ll n^{-B/4} \mbox{ and } \|v-u_{i_0}\|\ll n^{-B/4}.$$
\end{lemma}
 
 \begin{proof}[Proof of Lemma \ref{lemma:stable}] First, by Theorem \ref{theorem:main:perturbation}, we can assume that 
  
 \begin{equation}\label{nodal:separation}
| \lambda_i-\lambda_{j}| \ge n^{-B/2}, \forall i\neq j.
 \end{equation}

 Assume that  $v= \sum_{1\le i\le n} c_i u_i$, where $\sum_i c_i^2=1$ and  $u_1,\dots,u_n$ are the unit eigenvectors of $M_n$. Then 
 
 \begin{equation}\label{nodal:bound}
 \|(M_n-\lambda)v\| = (\sum_i c_i^2 (\lambda_i -\lambda)^2)^{1/2} \le n^{-B}.
 \end{equation}
 
 Let $i_0$ be an index with $c_{i_0}\ge 1/\sqrt{n}$, then $\|(M_n-\lambda)v\| \le n^{-B}$ implies that $|\lambda-\lambda_{i_0}| \le n^{-B+1/2}$, and so by \eqref{nodal:separation}, $|\lambda-\lambda_i|\ge n^{-B/2}/2$ for all $i\neq i_0$. Incorporating with \eqref{nodal:bound} we obtain
 
$$|c_i| =O( n^{-B/2}), \forall i\neq i_0.$$
 
 Thus  
 
 $$\|v-u_{i_0}\| \ll n^{-B/4}.$$
 \end{proof}

We next apply Theorem \ref{theorem:perturbation:poor} and Lemma \ref{lemma:stable} to prove Theorem \ref{theorem:perturbation:nodal}.


\begin{proof}[Proof of Theorem \ref{theorem:perturbation:nodal}] It suffices to prove for $i=1$. The eigenvalue equation $M_nv=\lambda v$ for with $|v_1|\le n^{-B}$ can be separated as

$$(M_{n-1}-\lambda)v' =-v_1 X; \mbox{ and } v'^T X=v_1m_{11},$$
 
where $X$ is the first column vector of $M_n$, and hence independent of $M_{n-1}$.  As $B$ is sufficiently large compared to $A$ and $\gamma$, and by adding a small amount of order $O(n^{-B})$ to the first component of $v'$ so as to $\|v'\|=1$, we can reduce the problem to 

$$\P\left( \exists \lambda\in \R, v'\in S^{n-2}: \|(M_{n-1}-\lambda)v'\|\le n^{-B/2}; \mbox{ and } |v'^T X|\le n^{-B/2}\right) =O( n^{-A}).$$

To this end, by Lemma \ref{lemma:stable}, if $\|(M_{n-1}-\lambda)v'\|\le n^{-B/2}$ then $M_{n-1}$ has a unit eigenvector $u'$ with $\|u'-v'\|\ll n^{-B/8}$. So $|v'^T X| \le n^{-B/2}$ would imply that 

\begin{equation}\label{nodal:smallball}
|u'^T X|\ll n^{-B/16}.
\end{equation}

However, by Theorem \ref{theorem:perturbation:poor}, \eqref{nodal:smallball} holds with probability $O(n^{-A})$. This concludes our proof.
\end{proof}

{\bf Acknowledgments.} The authors  thank S.~Ge for comments and corrections on an early version of this manuscript. They are also grateful to H.-T.~Yau for help with references.

\appendix

\section{Proof of Theorem \ref{theorem:delocalization} } \label{localization}

\begin{lemma}  \label{mass}  There are positive constants $\gamma_1,\gamma_2,\gamma_3$ (depending on the sub-gaussian moments) with  $\gamma_2<1$ such that the following holds with  probability at least $1 -\exp( -\gamma_1 n)$ with respect to the symmetric Wigner matrix $X_n$ as in Theorem \ref{theorem:main:Wigner}. 
For any unit  eigenvector $u =(u_1, \dots, u_n) $ of $X_n$ and any subset $S \subset \{1, \dots, n \}$ of size $\gamma_2n$,  

$$\sum_{i \in S } u_i^2 \le \gamma_3. $$ 

The same statement holds for $A_n (p)$.
\end{lemma}

Theorem \ref{theorem:delocalization} follows easily from this lemma. Indeed, consider a unit  eigenvector $u$ of $X_n$.
 Let $S$ be the set of coordinates  with absolute value
at least $ (1- \gamma_3) n^{-1/2} $. Then 

$$\sum_{i \in S} u_i^2  \ge 1  -  n \times (1-\gamma_3)^2 n^{-1} \ge \gamma _3. $$

By Lemma \ref{mass}, any set $S$ with this property should have size at least $\gamma_2 n$, with probability at least $1 -\exp(- \gamma_1 n)$. 
We obtain Theorem \ref{theorem:delocalization}  by setting  $c_1 =\gamma_1 $, $c_2 =\gamma_2$ and $c_3  = (1 -\gamma_3)$.

Lemma \ref{mass} was proved for $A_n (p)$ in \cite[Theorem 3.1]{DLL} in a different, but equivalent form. The proof extends with no difficulty 
to Wigner matrices.  We also refer the reader to \cite{OVW} for related results of this type.

\section{Proof of Theorem \ref{lemma:Wigner:key}}\label{section:Wigner:key}

The treatment here is based on \cite{vershynin}. For short, we will write $T_D$ instead of $T_{D,\kappa,\gamma,\alpha}$. The key ingredient is finding a fine net for $T_D$.

\begin{lemma}\label{lemma:Wigner:structure:1}
Let $n^{-c}\le \alpha \le c'/4$. For every $D\ge 1$, the level set  $T_D$ accepts a $O(\frac{\kappa}{\sqrt{\alpha}D})$-net $\CN$ of size 

$$|\CN|\le \frac{(CD)^n}{\left(\sqrt{\alpha n}\right)^{c'n/2}}D^{2/\alpha},$$

where $C$ is an absolute constant.

\end{lemma}

 Roughly speaking, the principle is similar to the proof of Lemma \ref{lemma:perturbation:approx} and Lemma \ref{lemma:ER:approx}. We will break $x\in T_D$ into disjoint subvectors of size roughly $m$, where $m=\lceil \alpha n \rceil$, and find appropriate nets for each. In what follows we will be working with $S^{m-1}$ first. 

\begin{definition}[Level sets] Let $D_0\ge c \sqrt{m}$. Define $S_{D_0}\subset S^{m-1}$ to be the level set

$$S_{D_0}:=\{x\in S^{m-1}: D_0 \le \LCD_{\kappa,\gamma}(x) \le 2D_0 \}.$$


\end{definition}

Notice that $\kappa=o(\sqrt{m})$ because $\kappa=n^{2c}$ and $m\ge n^{1-c}$. We will invokde the following result from \cite{rv-rec}.
\begin{lemma}\cite[Lemma 4.7]{rv-rec}\label{lemma:D_0:1}
There exists a $(2\kappa/D_0)$-net of $S_{D_0}$ of cardinality at most $(C_0D_0/\sqrt{m})^m$, where $C_0$ is an absolute constant.
\end{lemma}

As the proof of this result is short and uses important notion of $\LCD$, we include it here for the reader's convenience.

\begin{proof}[Proof of Lemma \ref{lemma:D_0:1}]
For $x\in S_{D_0}$, denote

$$D(x):= \LCD_{\kappa,\gamma}(x).$$

By definition, $D_0\le D(x)\le 2D_0$ and there exists $p\in \Z^m$ with 

$$\left\|x -\frac{p}{D(x)}\right\| \le \frac{\kappa}{D(x)} =O\left(\frac{n^{2c}}{n^{1-c}}\right) =o(1).$$ 

As $\|x\|=1$, this implies that $\|p\| \approx D(x)$, more precisely

\begin{equation}\label{eqn:LCD:net:1}
1- \frac{\kappa}{D(x)} \le \left\|\frac{p}{D(x)}\right\| \le  1+ \frac{\kappa}{D(x)}.
\end{equation}

This implies that 

\begin{equation}\label{eqn:LCD:net:2}
\|p\| \le (1+o(1)) D(x) <3D_0.
\end{equation}

It also follows from \eqref{eqn:LCD:net:1} that

\begin{equation}\label{eqn:LCD:net:3}
\left\|x -\frac{p}{\|p\|}\right\| \le \left\|x -\frac{p}{D(x)}\right\| + \left\| \frac{p}{\|p\|}(\frac{\|p\|}{D(x)} -1)\right\| \le 2\frac{\kappa}{D(x)} \le \frac{2\kappa}{D_0}.
\end{equation}

Now set 

$$\CN_0:=\left\{\frac{p}{\|p\|}, p\in \Z^m \cap B(0,3D_0)\right\}.$$

By \eqref{eqn:LCD:net:2} and \eqref{eqn:LCD:net:3}, $\CN_0$ is a $\frac{2\kappa}{D_0}$-net for $S_{D_0}$. On the other hand, it is known that the size of $\CN_0$ is bounded by $(C_0\frac{D_0}{\sqrt{m}})^m$ for some absolute constant  $C_0$.
\end{proof}


In fact we can slightly improve the approximations in Lemma \ref{lemma:D_0:1} as follows.

\begin{lemma}\label{lemma:D_0:2}
Let $c \sqrt{m} \le D_0 \le D$. Then the set $S_{D_0}$ has a $(2\kappa/D)$-net of cardinality at most $(C_0D/\sqrt{m})^m$ for some absolute constant  $C_0$ (probably different from that of Lemma \ref{lemma:D_0:1}).
\end{lemma}

\begin{proof}(of Lemma \ref{lemma:D_0:2}) First, by Lemma \ref{lemma:D_0:1} one can cover $S_{D_0}$ by $(C_0D_0/\sqrt{m})^m$ balls of radius $2\kappa/D_0$. We then cover these balls by smaller balls of radius $2\kappa/D$, the number of such small balls is at most $(O(D/D_0))^m$. Thus there are at most $(O(D/\sqrt{m}))^m$ balls in total.
\end{proof}

Taking the union of these nets as $D_0$ ranges over powers of two, we thus obtain the following.

\begin{lemma}\label{lemma:net:ND:1}
Let $D\ge c \sqrt{m}$. Then the set $\{x\in S^{m-1}: c\sqrt{m} \le \LCD_{\kappa,\gamma}(x) \le D \}$ has a $(2\kappa/D)$-net of cardinality at most 

$$(C_0D/\sqrt{m})^m \log_2 D,$$

for some absolute constant  $C_0$ (probably different from that of Lemma \ref{lemma:D_0:1} and Lemma \ref{lemma:D_0:2}).
\end{lemma}

We can also update the net above for $x$ without normalization.

\begin{lemma}\label{lemma:net:ND:2}
Let $D\ge c \sqrt{m}$. Then the set $\{x\in \R^m, \|x\|\le 1, c\sqrt{m} \le \LCD_{\kappa,\gamma}(x/\|x\|) \le D \}$ has a $(2\kappa/D)$-net of cardinality at most $(C_0D/\sqrt{m})^m D^2$ for some absolute constant $C_0$.
\end{lemma}
\begin{proof}[Proof of Lemma \ref{lemma:net:ND:2}]
Starting from the net obtained from Lemma \ref{lemma:net:ND:1}, we just need to $2\kappa/D$-approximate the fiber of $span(x/\|x\|)$ in $B_2^m$.
\end{proof}

We now justify our main lemma.

\begin{proof}[Proof of Lemma \ref{lemma:Wigner:structure:1}] We first write $x=x_{I_0}\cup \spread(x)$, where $\spread(x)= I_1\cup \dots \cup I_{k_0} \cup J$ such that $|I_k| =\lceil \alpha n \rceil$ and $|J|\le \alpha n$.  Notice that we trivially have 

$$|\spread(x)| \ge |I_1\cup \dots \cup I_{k_0}| = k_0 \lceil \alpha n \rceil \ge |\spread(x)|- \alpha n \ge c'n/2.$$ 

Thus we have 

$$ \frac{c'}{2\alpha} \le k_0 \le \frac{2c'}{\alpha}.$$

In the next step, we will construct nets for each $x_{I_j}$. For $x_{I_0}$, we construct trivially a $(1/D)$-net $\CN_0$ of size 

$$|\CN_0| \le (3D)^{|I_0|}.$$

For each $I_k$, as 

$$\LCD_{\kappa,\gamma}(x_{I_k}/\|x_{I_k}\|)\le \LCDhat_{\kappa,\gamma}(x) \le D,$$

by Lemma \ref{lemma:net:ND:2} (where the condition $\LCD_{\kappa,\gamma}(x_{I_k}/\|x_{I_k}\|)\gg \sqrt{|I_k|}$ follows from, say Theorem \ref{theorem:cilf-erdos}, because the entries of $x_{I_k}/\|x_{I_k}$ are all of order $\sqrt{\alpha n}$ while $\kappa=o(\sqrt{\alpha n})$), one obtains a $(2\kappa/D)$-net $\CN_k$ of size 

$$|\CN_k|\le  \left(\frac{C_0D}{\sqrt{|I_k|}}\right)^{|I_k|} D^2.$$ 

Combining the nets together, as $x=(x_{I_0},x_{I_1},\dots,x_{I_{k_0}},x_J)$ can be approximated by $y=(y_{I_0},y_{I_1},\dots,y_{I_{k_0}},y_J)$ with $\|x_{I_j}-y_{I_j}\|\le \frac{2\kappa}{D}$, we have

$$\|x-y\|\le \sqrt{k_0+1}\frac{2\kappa}{D} \ll \frac{\kappa}{\sqrt{\alpha} D}.$$

As such, we have obtain a $\beta$-net  $\CN$, where $\beta=O(\frac{\kappa}{\sqrt{\alpha} D})$,  of size 

$$|\CN| \le 2^n |\CN_0| |\CN_1| \dots |\CN_{k_0}| \le 2^n (3D)^{|I_0|} \prod_{k=1}^{k_0} \left(\frac{CD}{\sqrt{|I_k|}}\right)^{|I_k|} D^2.$$

This can be simplified to 

$$|\CN|\le \frac{(CD)^n}{\sqrt{\alpha n}^{c'n/2}}D^{O(1/\alpha)}.$$

\end{proof}

Before completing the proof of Lemma \ref{lemma:Wigner:key}, we cite another important consequence of Lemma \ref{lemma:smallball:regularized} (on the small ball estimate) and Lemma \ref{lemma:tensorization} (on the tensorization trick). 
\begin{lemma}\label{lemma:Wigner:1:forward} Let $x\in \Incomp(c_0,c_1)$ and $\alpha \in (0,c')$. Then for any $\beta \ge \frac{1}{c_0}\sqrt{\alpha} (\LCDhat_{\kappa,\gamma}(x,\alpha))^{-1}$ one has

$$\rho_{\beta \sqrt{n}}((X_n-\lambda_0)x) \le \left(\frac{O(\beta)}{\sqrt{\alpha}} +e^{-\Theta(\kappa^2)}\right)^{n-\alpha n}.$$
\end{lemma}

\begin{proof}[Proof of Lemma \ref{lemma:Wigner:key}] It suffices to show the result for the level set $\{x\in T_D\backslash T_{D/2}\}$. With $\beta = \frac{\kappa}{\sqrt{\alpha} D}$, the condition on $\beta$ of Lemma \ref{lemma:Wigner:1:forward} is clearly guaranteed,

$$ \frac{\kappa}{\sqrt{\alpha} D} \gg \frac{\sqrt{\alpha}}{D}.$$

 This lemma then implies 

\begin{align*}
\P\left(\exists x\in T_D\backslash T_{D/2} : \|(X_n-\lambda_0) x -u\| =o (\beta \sqrt{n})\right) &\le \left(\frac{O(\beta)}{\sqrt{\alpha}} +e^{-\Theta(\kappa^2)}\right)^{n-\alpha n} \times \frac{(CD)^n}{(\sqrt{\alpha n})^{c'n/2}}D^{O(1/\alpha)}\\
& \le \left(\frac{O(\kappa)}{\alpha  D} +e^{-\Theta(\kappa^2)}\right)^{n-\alpha n} \times \frac{(CD)^n}{(\sqrt{\alpha n})^{c'n/2}}D^{O(1/\alpha)}\\
&\le n^{-c'n/8},
\end{align*}
provided that $\kappa =n^{2c}$ with sufficiently small $c$ compared to $c'$, and that $D \le n^{c/\alpha}$.
\end{proof}


\end{document}